\documentclass[letter,12pt]{article}
\usepackage[T1]{fontenc}
\usepackage{verbatim,comment}
\usepackage{amsmath,amsthm}
\usepackage{xspace}
\usepackage{pifont}
\usepackage{graphicx}
\usepackage{amssymb}
\usepackage{epic, eepic}
\usepackage{dsfont}
\usepackage{amssymb}
\usepackage{makeidx}
\usepackage{mathrsfs,enumerate}
\usepackage{amsmath,afterpage}
\usepackage{epsf}
\usepackage{graphics,color}
\usepackage[ampersand]{easylist}
\usepackage[utf8]{inputenc}
\usepackage{color}

\usepackage{hyperref}

\newtheorem{theorem}{Theorem}[section]
\newtheorem{lemma}[theorem]{Lemma}
\newtheorem{proposition}[theorem]{Proposition}

\newtheorem{corollary}[theorem]{Corollary}

\newtheorem{remark}[theorem]{Remark}

\newtheorem{definition}[theorem]{Definition}
\newtheorem{assumption}[theorem]{Assumption}
\numberwithin{equation}{section}
\newcommand{\hatd}[1]{{}}

\newcommand{\norm}[1]{\left\lVert#1\right\rVert}
\newcommand{\abs}[1]{\left\lvert#1\right\rvert}

\newcommand{\E}{E}
\newcommand{\F}{\mathcal{F}}

\newcommand{\Rd}{\mathbb{R}^{d}}

\newcommand{\less}{\;\le\;}

\newcommand{\foralll}{\quad\text{for all}\quad}

\newcommand{\for}{\quad\text{for}\quad}

\newcommand{\with}{\quad\text{with}\quad}
\newcommand{\aand}{\quad\text{and}\quad}
\newcommand{\eps}{\varepsilon}

\newcommand{\rr}{\rightarrow}

\newcommand{\ol}{\overline}
\newcommand{\wt}{\widetilde}

\newcommand{\bd}{\begin{displaymath}}
\newcommand{\ed}{\end{displaymath}}
\newcommand{\be}{\begin{equation}}
\newcommand{\ee}{\end{equation}}
\newcommand{\bq}{\begin{eqnarray}}
\newcommand{\eq}{\end{eqnarray}}
\newcommand{\bn}{\begin{eqnarray*}}
\newcommand{\en}{\end{eqnarray*}}
\newcommand{\dl}{\delta}
\newcommand{\re}{\mathbb{R}}

 \newcommand{\lam}{\lambda}

\newcommand{\R}{\mathbb{R}}

\usepackage{xcolor}

\usepackage{authblk}

\title{}
\title{The Multiplicative Chaos of $H=0$ Fractional Brownian Fields }
\author[1]{Paul Hager }
\author[2]{Eyal Neuman  }
\affil[1]{Institute of Mathematics, Technische Universität Berlin}
\affil[2]{Department of Mathematics, Imperial College London }

\begin{document}

 \vspace{-0.5cm}
\maketitle

\begin{abstract}
We consider a family of fractional Brownian fields $\{B^{H}\}_{H\in (0,1)}$ on $\mathbb{R}^{d}$, where $H$ denotes their Hurst parameter. We first define a rich class of normalizing kernels $\psi$ such that the covariance of 
$$
X^{H}(x) = \Gamma(H)^{\frac{1}{2}} \left( B^{H}(x) - \int_{\mathbb{R}^{d}} B^{H}(u) \psi(u, x)du\right),
$$
converges to the covariance of a log-correlated Gaussian field when $H \downarrow  0$.   

We then use Berestycki's ``good points'' approach \cite{berestycki2015elementary}  in order to derive the limiting measure of the so-called \emph{multiplicative chaos of the fractional Brownian field}
$$
M^{H}_\gamma(dx) = e^{\gamma X^{H}(x) - \frac{\gamma^{2}}{2} E[X^{H}(x)^{2}] }dx, 
$$
as $H\downarrow 0$ for all $\gamma \in (0,\gamma^{*}(d)]$, where $\gamma^{*}(d)>\sqrt{\frac{7}{4}d}$. As a corollary we establish the $L^{2}$ convergence of $M^{H}_\gamma$ over the sets of ``good points'', where the field $X^H$ has a typical behaviour. As a by-product of the convergence result, we prove that for log-normal rough volatility models with small Hurst parameter, the  volatility process is supported on the sets of ``good points'' with probability close to $1$. Moreover, on these sets the volatility converges in $L^2$ to the volatility of multifractal random walks.

\end{abstract}

\medskip

\noindent {\it 1991 Mathematics Subject Classification}: Primary, 60G15; 60G57; 60G60; Secondary, 60G18.

\noindent {\it Keywords and phrases}: fractional Brownian fields, log-correlated Gaussian fields, rough volatility, Gaussian multiplicative chaos, multifractal random walk.
\bigskip

\tableofcontents

 \section{Introduction}\label{sec:introduction}
 We consider a class of Gaussian fields which is known as \emph{fractional Gaussian fields} (FGF). We study the phase transition between two sub-classes of random fields therein, which are called \emph{fractional Brownian fields} (FBF) and \emph{log-correlated Gaussian fields} (LGF).
 
The $d$-dimensional factional Gaussian field $h$ on $\re^d$, with index $s \in \re$  (often referred to as $\textrm{FGF}_s(\re^d))$ is formally defined as 
 $$
 h = (-\Delta)^{-s/2} W,  
 $$
 where $W$ is a white noise on $\re^d$ and $ (-\Delta)^{-s/2}$ is the fractional Laplacian in $\re^d$.
 For a rigorous definition we refer to the survey paper \cite{Lod-She-Sun-Wat16}.
 
 The class of FGFs has attracted considerable attention in recent years as it includes some well known Gaussian processes and Gaussian fields which arise from the areas of stochastic analysis, mathematical physics and financial modeling. When $d=1$ and $s=1$, $h$ is a Brownian motion. The case where $s=0$ coincides with white noise and the case where $s=1$ is the Gaussian free field (GFF), both on $\re^d$.  
 
 It is often convenient to refer to the Hurst parameter
 $$
 H:=s-\frac{d}{2}, 
 $$
 that describes the scaling relations of FGFs.  For  $h\sim \textrm{FGF}_s(\re^d)$ we have
 $$
 h(\alpha \cdot)  \stackrel{d}{=}  \alpha ^H h(\cdot),  \quad \textrm{for all } \alpha>0.
 $$
 In the case where $h$ is a random tempered distribution this relation is described by using test functions (see Section 1 of \cite{Lod-She-Sun-Wat16}).
 
The \emph{fractional Brownian field} (FBF) with a Hurst parameter $H \in (0,1)$,
is a zero-mean Gaussian field $(B^H(x))_{x \in \mathbb{R}^{d}}$ with a covariance kernel given by
\be \label{fbf}
E[B^H(x) B^H(y)]= \frac{1}{2} \left(\|x\|^{2H} + \|y\|^{2H} -\|x-y\|^{2H} \right), \quad x,y \in \re^{d}, 
\ee
where $\| \cdot\|$ denotes the Euclidean norm.  This random field was introduced by Yaglom \cite{Yaglom} as a model for turbulence in fluid mechanics. 
The validity of the covariance kernel and properties of these fields such as series expansions and functional central limit theorems were extensively studied in \cite{Gangolli,Ossiander,samoradnitsky2017stable,Malyarenko,Cohen,herbin2006,lindstrom1993fractional} among others.
The case where $d=1$ is the well known \emph{fractional Brownian motion} (FBM) (see \cite{mandelbrot1968fractional}) which is a very popular modeling object in many fields such as hydrology \cite{molz1997fractional}, telecommunications and network
traffic \cite{leland1994self,mikosch2002network} and finance \cite{comte1998long}. It is shown in Section 6 of \cite{Lod-She-Sun-Wat16} that the FBF coincides with the  $\textrm{FGF}_s(\re^d)$ for $H = s-\frac{d}{2} \in (0,1)$ after choosing a suitable representation of the latter by a continuous field.

Another class of random fields which is in the focus of this work is the \emph{log-correlated Gaussian fields} (LGF).
We denote by $\mathcal S$ the Schwartz space of smooth functions on $\re^d$ with rapid decay and by $\mathcal S'$ the topological dual space of tempered distributions.
Further denote by $\mathcal{S}_0 \subset \mathcal{S}$ the space of mean-zero test functions.
The log-correlated Gaussian field $X$ is a centred Gaussian field in the space of tempered distributions modulo constants $\mathcal{S}^{\prime}/\mathcal{S}_0$ with the following covariance structure 
$$
E[ \langle X, \phi_1 \rangle  \langle X, \phi_2 \rangle] = \int_{\re^d} \int_{\re^d}\log \frac{1}{\|x-y\|} \phi_1(x) \phi_2(x)dx dy, \quad \phi_1,\phi_2\in \mathcal{S}_0. 
$$
When fixing the constants of the filed, e.g. by "pinning the field down" at a specific test function, the covariance kernel changes by an additional bounded function.  
Further variants of this definition such as choosing a different metric space as the underlying domain are also studied extensively in the literature (see e.g. Section 2 of \cite{Rhodes-Vargas14}). 

In $d=1$ the LGF was proposed as a financial model for the log-volatility \cite{Bacry13, duchon2012forecasting}. In $d=2$ the LGF coincides with the GFF up to a multiplicative constant factor. Other physical applications are also available for LGFs on higher dimensions (see Section 1.1 \cite{Dup-Rho-She-Var2014}). 
Moreover it was shown in Section 3 of \cite{Lod-She-Sun-Wat16} that the $d$-dimensional LGF is a multiple of $\textrm{FGF}_{d/2}(\re^d)$, that is, it is formally an $H=0$ fractional Gaussian field. 

Since both the FBF and the LGF are embedded in the class of fractional Gaussian fields, taking the limit as $H \rr 0$ on a sequence of FBFs formally gives a phase transition within the FGF class. We refer to Figure 1.2 in \cite{Lod-She-Sun-Wat16} for an illuminating phase transitions diagram between various subclasses of FGFs. However, plugging in directly $H=0$ in the covariance function \eqref{fbf} does not lead to any relevant process. 

Several authors have already defined some fractional Brownian motions with ${H=0}$, see in particular \cite{FKS16}. This is usually done through a regularization procedure. In \cite{neuman2018fractional} a different approach was taken for the $d=1$ case (i.e. for the limit as $H \rr0$ of FBM). The process $B^{H}$ was normalized in order to get a non-degenerate limit. The normalized sequence of processes $(X^{H}_{.})_{H\in(0,1)}$ was defined as follows:
\be \label{nr-norm}
X^{H}_{t}=\frac{B_{t}^{H}-\frac{1}{t}\int_{0}^{t}B^{H}_{u}\,du}{\sqrt{H}}, \quad  t\in\mathbb{R},
\ee
where $X_{0}^{H}=0$. Subtracting the integral in the numerator and dividing by $\sqrt{H}$ enables us to get a non-trivial limit for our sequence as $H$ tends to $0$. The approach in \cite{neuman2018fractional} was quite simple and natural from the financial viewpoint, as the normalized processes remains adapted. The main result in \cite{neuman2018fractional} states  that the sequence $\{X_{t}^{H}\}_{t \in \re}$ converges weakly as $H$ tends to zero, towards a centered Gaussian field $X$ satisfying for any $\phi_{1},\phi_{2}\in \mathcal S$
$$E[\langle X^{}_{}, \phi_{1} \rangle^{}\langle X^{}_{}, \phi_{2} \rangle^{}] =  \int_{\re}\int_{\re} K(t,s) \phi_{1}(t)\phi_{2}(s)\,dt\,ds, 
$$
where for $-\infty<s,t<\infty$, $s\neq t$ and $s,t\neq 0$ $$
 K(t,s)=\log\frac{1}{|t-s|} +g(t,s), 
$$
and $g$ is a bounded function for $s,t$ away from zero. 

One of the main objectives of this work is to extend the results in \cite{neuman2018fractional} to $\re^d$. We construct a sequence of normalized FBFs in $\mathbb R^{d}$ that converges to a LGF. We also generalise the result in \cite{neuman2018fractional} in the sense that we characterise the class of normalizing processes which lead to a meaningful limit as $H \rightarrow 0$. We show that the normalized process in \eqref{nr-norm} is just one member of the class of normalizing processes which inherits the self-similarity from the FBF.
We will also give example for a class of normalizations which preserve the stationarity of increments.

The construction of the \emph{Gaussian multiplicative chaos} (GMC) associated to a LGF $X$, is a random measure that is formally given by
$$
M_\gamma(dx) = e^{\gamma X(x) - \frac{\gamma^{2}}{2}\E(X(x)^{2}) }dx,
$$
for $\gamma > 0$. This measure was first introduced by Kahane \cite{kahane85} and later generalized in \cite{Rhodes-Vargas14,vargas16,Robert-Vargas10} and the references therein. The GMC has an extensive use in finance, as we discuss later, and also in turbulence \cite{Che-Rob-Var10,Fyo-Dou-Ros10}, disordered systems \cite{FyoBcu08,madaule2016} and Liouville quantum gravity \cite{Rhodes-Vargas14,vargas16}.

Since $X$ is a tempered distribution, one usually uses a smooth local mollifying function $\theta_{\eps}$ which converges to Dirac's delta measure as $\eps \rr0$. Then define $X^{\eps} = X\star \theta_{\eps} (x)$, where $\star$ denotes the convolution operation. For a given domain $D \subset \re^d$, define for any Borel measurable $A \subset D$ the approximation
\be \label{gmc-m}
M^{\eps}_\gamma(A) = \int_{A} e^{\gamma X^{\eps}(x) - \frac{\gamma^{2}}{2}\E(X^{\eps}(x)^{2}) }dx. 
\ee
It is well known that for $\gamma <  \sqrt{2d}$ this measure converges weakly to a non-degenerate limiting measure, which is called the GMC associated to $X$. We refer to a review paper by Rhodes and Vargas \cite{Rhodes-Vargas14} for additional details.  

The convergence in probability of  
\begin{align}  \label{vol-h}
M^{H}_\gamma([0,t]) = \int_{[0,t]} e^{\gamma X^{H}(s) - \frac{\gamma^{2}}{2}\E(X^{H}(s)^{2}) }ds, \quad t \in \re_+, 
\end{align}
when $H \downarrow 0$ towards a Gaussian Multiplicative Chaos (GMC) was proved in Corollary 2.2 of \cite{neuman2018fractional}.  
However the proof in  \cite{neuman2018fractional} is indirect, as it uses a dominance argument between the covariance of $X^{H}$ and the covariance of a ``standard'' kernel, for which the convergence properties are known. The conclusion is also based on a general result from the theory of randomized shifts by Shamov \cite{shamov16}. 

In Theorem \ref{thm-convergence} of this paper we derive a stronger convergence statement, which also applies in $\Rd$. Using Berestycki's elementary and self contained approach \cite{berestycki2015elementary}, we show that for small values of $H$, $M^{H}_\gamma(\cdot)$ vanishes on the complement of the so called "good points" of the measure, with a high probability. On the good points set we prove the $L^{2}$ convergence of $M^{H}_\gamma(\cdot)$ as $H$ tends to $0$ (see Corollary \ref{corr-supp}, Proposition \ref{prop-l2}, and the explanation at the beginning of Section \ref{section-conv}). 
The proof of convergence is direct and transparent, and it improves our understanding of the transition from stochastic exponential of a fractional Brownian fields to GMC. 

These improved results shed new light on the properties of the support of rough-volatility models with small a Hurst parameter and also show that volatility process on the set of good points convergence in $L^2$ to the volatility of multifractal random walks. We discuss these applications in more detail in Section \ref{fin-mot}.  

Moreover, the convergence result in Theorem \ref{thm-convergence} does not explicitly depend on the construction of the fractional Brownian fields, but only on their cross-covariance structure, which is defined in \eqref{eq:XHCovariance}.
This is in contrast to the convergence result of Shamov in Theorem 25 of \cite{shamov16}, which explicitly imposes conditions on the construction of the fields approximating the LGF.
The class of fractional Brownian fields for which Theorem \ref{thm-convergence} is applicable, includes all normalizations of fractional Brownian fields as discussed in Section \ref{sec-norm}, however other examples are conceivable. 

\subsection{Financial Motivation}  \label{fin-mot}
Modeling the volatility of assets returns using factional Brownian motion trace back to the pioneering work of Comte \cite{comte1998long}. 
Recently, a new approach has been introduced in \cite{Gat-Jai-Ros14} for the use of FBM with small Hurst parameter in volatility modelling . Careful analysis of volatility process of thousands of assets suggests that the log-volatility process actually behaves like a FBM with Hurst parameter between 0.02 to 0.2 ( see \cite{bennedsen2016decoupling} and \cite{Fukas19}). Hence various approaches using FBM with small Hurst parameter have been introduced for volatility modeling. These models are referred to as {\it rough volatility models}, see \cite{bayer2017regularity,bayer2016pricing,bayer2017short, bennedsen2017hybrid,el2016characteristic,forde2015asymptotics,fukasawa2017short,jacquier2017pathwise} for more details and practical applications.

Another class of popular models for assets returns is the \emph{multifractal random walks}  (see e.g \cite{Man-Cal-Fis97, bacry2001multifractal,Muz-Bac03, Muz-Bac03, Bar-Men02, cal-fis04}, among others). In these models the log-price is defined as $Y_t=B_{M([0,t])},$
where $B$ is a Brownian motion and $$M(t)=\underset{l\rightarrow 0}{\text{lim }}\sigma^2\int_0^t e^{w_l(u)}du, \text{ a.s.},$$ with $\sigma^2>0$ and
$w_l$ a Gaussian process such that for some $\lambda^2>0$ and $T>0$ 
$$\text{Cov}[w_l(t),w_l(t')]=\lambda^2\text{log}(T/|t-t'|), \text{ for }l<|t-t'|\leq T.$$   
We refer to \cite{Muz-Bac03} for additional details. Hence we see that as $l \rr 0$, $M$ formally corresponds to a measure of the form $\text{exp}(X_t)dt$, where $X$ is a LGF. This again could be made rigorous by using the notion of Gaussian multiplicative chaos which was described earlier. 

One of the main goals, and in fact the initial motivation of writing this paper is to describe the phase transitions of the volatility process between rough volatility models, which are indexed by a Hurst parameter $H>0$, and the multifractal random walk model which corresponds to $H=0$. In particular we would like to classify the class of processes that can be used to normalize $B^{H}$ as in \eqref{nr-norm}, as the current normalization is quite specific and keeps $X_t^H$ as a self-similar process. Having a large class of suitable normalizing processes could help us to choose $X^H$ which fits time series observations better (see Section \ref{sec-norm} for additional details). 

We also derive the convergence of the volatility $M^H_\gamma([0,t])$ in \eqref{vol-h} when $H$ tends to $0$. Theorem \ref{thm-convergence} in the one-dimensional case improves the convergence in probability result of \cite{neuman2018fractional}. We provide a stronger and more refined statement by showing that for small values of $H$, $M^{H}_\gamma([0,t])$ vanishes outside sets of "good points", with a high probability. On the sets of "good points", where $X^H$ experience a typical behaviour,  we prove the $L^{2}$ convergence of $M^{H}_\gamma$ as $H$ tends to $0$ (see Corollary \ref{corr-supp} and Remark \ref{rem-supp} afterwords, Proposition \ref{prop-l2} and explanation at the beginning of Section \ref{section-conv}). These refined results point out that the volatility process in log-normal rough volatility models with small (but not necessarily zero) Hurst parameter, are supported on the sets of good points with probability close to $1$.  Moreover, it follows that on the good points sets, the rough-volatility process converges in $L^2$ to the volatility of multifractal random walks.

\section{Main results}
In this section we present our main results on the convergence of the stochastic exponential of FBFs when $H$ tends to zero, and on the normalization of the FBFs. 
We first present our convergence results. 

\subsection{Convergence of the Multiplicative Chaos of FBFs}  

Let $D$ be a bounded domain in $\re^d$ and fix $H_0 \in (0,1/2)$. We call $X = (X^{H})_{0 < H < H_0}$ a family of normalized fractional Brownian fields if it has the following covariance structure
\begin{equation}\label{eq:XHCovariance}
\E(X^{H}(x)X^{h}(y)) = C_{H, h}\left( \frac{1- \norm{x-y}^{H+h}}{H+h} + g^{H,h}(x,y) \right),
\end{equation}
for $x,y \in D$ and $H,h \in (0, H_0)$, where $C_{H,h} > 0$ is a constant and $g^{H,h}: D \times D \to \R$ is a bounded function.
For the rest of this paper we will make the following assumption. 
\begin{assumption}\label{asmp}   
We assume that covariance function in (\ref{eq:XHCovariance}) satisfies the following: 
\begin{enumerate}
\item  
\begin{align} \label{c-h}
\lim_{\ol H\to 0}\Big( \sup_{0<h,H < \ol H}\abs{C_{h, H} -1}\Big) = 0.
\end{align}
\item  The functions $g^{H,h}$ converge to a bounded function $g$ as follows 
\begin{equation} \label{eq:gEstimate}
\lim_{(H,h) \rr 0} \sup_{x,y\in D}\big| g^{H,h}(x,y) - g(x,y)\big| =0,
\end{equation}
where the limit $(H,h) \rr 0$ is understood as a limit in $\re^2$. 
\end{enumerate} 
\end{assumption}

\begin{remark} \label{rem-cov} 
From Assumption \ref{asmp} we get the pointwise convergence of the covariance kernel 
\begin{align}\label{eq:limit_XH_covariance}
\lim_{H\to 0}\E(X^{H}(x)X^{H}(y)) = \log\frac{1}{\norm{x-y}} + g(x,y), \quad \textrm{for all } x,y \in D, \ x\neq y.
\end{align}
where $g$ is given in \eqref{eq:gEstimate}. 
\end{remark}

\begin{remark}
In Section \ref{sec-norm} we will show that a family of FBFs constructed on the same Wiener space with suitable normalization, has the cross-covariance structure \eqref{eq:XHCovariance} and satisfies Assumption \ref{asmp}.
In particular this shows that the normalized process from \cite{neuman2018fractional}, which is given in \eqref{nr-norm}, is included in the class of normalized FBFs which satisfy Assumption \ref{asmp}. 
\end{remark} 

For $\gamma > 0$ and every $H\in (0, H_0)$ we define the random measure $M^{H}_\gamma$ on $D$ as follows 
\begin{align}\label{approxMeasFBM}
M^{H}_\gamma(dx) = e^{\gamma X^{H}(x) - \frac{\gamma^{2}}{2}\E(X^{H}(x)^{2}) }dx. 
\end{align}
We call $M^{H}_\gamma$ as the multiplicative chaos associated to the normalized FBF $X^H$.

Now we are ready to present one of our main result which deals with the convergence in probability of $M^{H}_\gamma$ as $H \rr 0$.

\begin{theorem}\label{thm-convergence}
The sequence of measures $\{M^{H}_\gamma\}_{H \in (0, H_0)}$ converges in probability as $H\rr0$ to a Borel measure $M_\gamma$ in the topology of weak convergence of measures on $D$, for all $\gamma \leq \gamma^{*}(d)$, with $\gamma^{*}(d)>\sqrt{\frac{7}{4}d}$.
\end{theorem}

\begin{remark} 
Theorem \ref{thm-convergence} generalizes Corollary 2.2 of \cite{neuman2018fractional} to any dimension. We recall that the later dealt with convergence in probability of $M^{H}_\gamma$ on $\re$. A central ingredient in the the proof relates to the concept of ``good points'', which are points in the domain $D$ where the field $X^H$ has a typical behaviour (see \eqref{g-set} for the precise definition). The proof of Theorem \ref{thm-convergence} derives a stronger statement of convergence than Corollary 2.2 of \cite{neuman2018fractional}, as we show that for small $H$, $M^{H}_\gamma(\cdot)$ vanishes on the complement of the set of good points with high probability (see Corollary \ref{corr-supp} and Remark \ref{rem-supp}). 
On the set of good points we prove the $L^{2}$ convergence of $M^H_\gamma$ (see Proposition \ref{prop-l2} and the discussion at the beginning of Section \ref{section-conv}). 
\end{remark} 

\begin{remark} 
The proof of Theorem \ref{thm-convergence} is based on Berestycki's approach for the  construction of Gaussian multiplicative chaos \cite{berestycki2015elementary}. As we mentioned before, this was done by first mollifying the log-correlated Gaussian field $X^\eps$ in \eqref{gmc-m}, which corresponds to $X^H$ in our case. However $X^\eps$ as being a mollified version of LGF has some nice properties which are fundamental for the proof. For instance, if $\eps = e^{-t}$ and $\wt X_t:= X^\eps(x)$, then we have
$$
\textrm{Cov}(\wt X_t, \wt X_s) = s\wedge t +O(1),  
$$
see equation (3.2) and Lemma 3.5 therein. This means that on the scale of $\eps = e^{-t}$, $X^\eps$ behaves approximately like a Brownian motion. This property clearly does not apply for $\{X^H\}_{H \in (0,\ol H)}$ which experience the long range dependence of the fractional Brownian field. 

In order to overcome this gap, we had to improve Berestycki's argument to our purpose in several parts of the paper. For example, in the proof of Proposition \ref{prop-j1} we bound the two point probabilities (see \eqref{p-med} and Lemma \ref{lem-p-star}), where in \cite{berestycki2015elementary} one point probability was sufficient (see (3.8)-- (3.11) therein). This bound was crucial both for the proof of uniform integrabilty (see Proposition \ref{prop:uniform_integrability}) and for the proof of convergence (see Lemmas \ref{lemma-con1} and \ref{lem-lb}). The  improvements helped to enlarge the convergence interval in Theorem \ref{thm-convergence}, however we did not get the full convergence interval as in Theorem 1.1 of \cite{berestycki2015elementary} which was $\gamma <\sqrt{2d}$. The question whether $M^H_\gamma$ converge when $H\rr 0$, for $ \gamma^*(d) \leq \gamma < \sqrt{2d}$ remains as an interesting open  question. 
\end{remark} 

\begin{remark} [Application to rough-volatility models]
Corollary \ref{corr-supp} which is one of the ingredients in the proof of Theorem \ref{thm-convergence}, gives novel results on the properties the support of rough-volatility models with small Hurst parameter.
Indeed we show that for $H$ small enough the  volatility process $M^H_\gamma([0,t])$ in \eqref{vol-h} is supported on good points $G^{H,\ol H}_\alpha(x)$ in \eqref{g-set}, with a probability that is asymptotically close to $1$. Moreover, from Proposition \ref{prop-l2} it follows that on the good points set the rough-volatility process converges in $L^2$ to the volatility of multifractal random walks (see also Remark \ref{rem-supp}). 
\end{remark}

\subsection{Normalization of fractional Brownian fields} \label{sec-norm}
In this section we will define a general class of normalizations in the sense of \eqref{nr-norm}, that apply to fractional Brownian fields. The normalized field will be a centered Gaussian field with covariance as in \eqref{eq:XHCovariance} that satisfies Assumption \ref{asmp}. 

We first explain how we construct the family of FBFs $\{B^H\}_{H \in (0,H_0)}$ for some $0<H_0<1$, on the same probability space. Then we present our main results regarding the normalization.  The values of the constants $b_{h,H}$, $h_{h,H}$, $o_{h,H}$, $m_H$, $k_H^d$, $C_{H,h}^d$ that appear in this section are given in Appendix \ref{appendix}.  

We start by construction in the FBM case, i.e. when $d=1$.
Let $(\Omega, \F, (\F_t)_{t \in \R}, \mathbb{P})$ be a filtered probability space on which a two-sided standard Brownian motion $W=(W_t)_{t\in\R}$ is defined.
A well known result by Mandelbrot and van Ness \cite{mandelbrot1968fractional}
states that the following stochastic integral
\begin{align}\label{eq:mandelbrot_van_ness}
\wt B^{H}(t) = m_H \int_{\mathbb{R}}\Big( (t-s)^{H-\frac{1}{2}}_+ - (-s)^{H-\frac{1}{2}}_+\Big) dW_s, \quad t\in\R,
\end{align}
defines a fractional Brownian motion $\wt B^{H} = (\wt B^{H}(t))_{t\in\R}$ with a Hurst parameter $H\in(0,1)$.
Moreover, it is evident from this construction, yet rarely considered in the literature, that (\ref{eq:mandelbrot_van_ness}) induces a cross-correlation for fractional Brownian motions of different Hurst parameters.
In particular it follows from Theorem 2 in \cite{dobric2006fractional} that 
\be \label{mv-cov}
\begin{aligned} 
\E(\wt B^{H}(t) \wt B^{h}(s)) 
&=\Bigg. b_{h,H}\cdot\Big(\abs{s}^{h+H} + \abs{t}^{h+H} - \abs{t-s}^{h+H}\Big) - o_{h,H}\cdot f^{h,H}(s,t),
\end{aligned}
\ee
for all $t, s \in \R$ and $H, h \in (0,1)$ with $H + h\neq 1$ where 
\be\label{mv-f}
\begin{aligned} 
f^{h,H}(s,t) = \mathrm{sgn}(s)\abs{s}^{h+H} + \mathrm{sgn}(t)\abs{t}^{h+H} - \mathrm{sgn}(t-s)\abs{t-s}^{h+H}.
\end{aligned}
\ee
The Mandelbrot-van Ness representation is particularly interesting for financial applications, since $(\wt B^{H}(t))_{t\ge 0}$ is adapted to the filtration $(\F_t)_{t\ge 0}$ and therefore allows to construct an adapted fractional Brownian motion with correlation to the underlying Brownian motion $W$.
The construction (\ref{eq:mandelbrot_van_ness}) is sometimes referred to as \emph{non-anticipating}.

Note that there is no evident extension of (\ref{eq:mandelbrot_van_ness}) to the construction of fractional Brownian fields on $\Rd$.
However, we can give up the adaptedness and replace the kernel in  (\ref{eq:mandelbrot_van_ness}) by a reflected version, defining a fractional Brownian motion $B^{H}$ for $H \in (0,1)\setminus\{\frac{1}{2}\}$ by
\begin{align}\label{eq:one_dim_symmetric}
 B^{H}(t) = k^{1}_H \int_{\mathbb{R}}\Big( \abs{t-s}^{H-\frac{1}{2}} - \abs{s}^{H-\frac{1}{2}}\Big) dW_s, \quad t\in\R.
\end{align}
This representation is sometimes referred to as \emph{well-balanced} (see \cite[Chapter 7.2.1]{samoradnitsky2017stable}).
In Proposition 11 of \cite{dobric2006fractional} it was shown that 
\begin{align} \label{cov-mv}
E [ B^{H}(t) B^{h}(s)] = c^{1}_{H,h}\Big(\abs{s}^{h+H} + \abs{t}^{h+H} - \abs{t-s}^{h+H}\Big), \quad  t,s \in \re. 
\end{align}

The construction (\ref{eq:one_dim_symmetric}) has a natural extension to fractional Brownian fields.
Let $W$ be a \emph{white noise measure} in $\Rd$, defined on a probability space $(\Omega, \F, \mathbb{P})$.
It was proved in \cite{lindstrom1993fractional} that we can construct a fractional Brownian field $B^{H}$ by
\begin{align}\label{eq:fbm_construction_rd}
B^{H}(x) = k^{d}_{H}\int_{\Rd}\Big( \norm{x-y}^{H-\frac{d}{2}} - \norm{y}^{H-\frac{d}{2}}\Big) W(dy), \quad x\in\Rd,
\end{align}
where $H \in (0,1)$. Since we could not find a reference for the computation of covariance of \eqref{eq:fbm_construction_rd}, we derive this result in the following lemma. 

\begin{lemma}[Covariance of fractional Brownian fields]\label{lem:proof_fbf_covariance}  The covariance structure of $\{B^H\}_{H\in (0,1)}$ in \eqref{eq:fbm_construction_rd} given by
\begin{align} \label{cov-fbf}
E[B^{H}(x)B^{h}(y)]=c_{H,h}^d\big(\norm{x}^{H+h}+\norm{y}^{H+h}+\norm{x-y}^{H+h}\big), 
\end{align}
for all $x,y \in \Rd$, $H, h \in (0,1/2)$.
\end{lemma}
The proof of Lemma \ref{lem:proof_fbf_covariance} is postponed to Section \ref{corr-comp}. 

Next we define the class of normalizing functions for the FBFs which were described above.  

Let $\psi$ be a positive integration kernel on $\Rd$, that is $\psi: \Rd \times \Rd \to \R_+ $ is a measurable function. 
For a domain $D\subset{\Rd}$ and $0 < H_0 < \frac{1}{2}$ we define the following class of normalizing kernels.

\begin{definition} [Normalizing kernels]  \label{def-ker}
We say that the kernel $\psi$ is in the class of normalizing kernels $\mathcal N_{H_0}(D)$, if it satisfies the following conditions: 
\begin{itemize}
\item[(i)] For any $y\in D$, $y\mapsto \psi(x,y)$ is continuous almost everywhere and  
\begin{align}
\label{psiConditionOne}
\int_{\re^d}\psi(x,y)dx &= 1.
\end{align}
\item [(ii)] The following bounds hold:
\begin{align} \label{psiConditionTwo}
 \sup_{y\in D}\int_{\Rd}\norm{x}^{2H_0}\psi(x,y) dx & < \infty, \\
\label{psiConditionThree}
\sup_{y\in D} \int_{\Rd}\left(\log_- \norm{x-y} \right)^2 \psi(x,y) dx & < \infty,\\
\label{psiConditionFour}
\sup_{y,w \in D} \int_{\Rd}\int_{\Rd}\left(\log_-\norm{x-v}\right)^2 \psi(x,y) \psi(v,w) dxdv &< \infty,
\end{align}
\end{itemize}
where $\log_-(x) = \min(\log x, 0)$.
\end{definition} 
 
Now we are ready to state our main result regarding FBFs normalization. 

\begin{theorem}\label{thm:fbm_normalization}
Let $\{B^{H}\}_{H \in (0, H_0)}$ be a family of fractional Brownian fields as in \eqref{eq:mandelbrot_van_ness}, \eqref{eq:one_dim_symmetric} or \eqref{eq:fbm_construction_rd}.  
For any $\psi \in \mathcal N_{H_0}(D)$, 
\begin{align}\label{def:XH}
X^{H}(x) := \Gamma(H)^{\frac{1}{2}} \left( B^{H}(x) - \int_{\Rd} B^{H}(u) \psi(u, x)du\right), \quad x \in D, \ H\in(0,H_0),
\end{align}
is a family of centred Gaussian fields with covariance structure \eqref{eq:XHCovariance}, which agrees with Assumption \ref{asmp}.  
\end{theorem}

\begin{remark}
From Remark \ref{rem-cov} and Theorem \ref{thm:fbm_normalization} it follows that the covariance of $X^{H}$ converges pointwise to the covariance of the LGF as $H \rr 0$. We recall the Lévy-continuity theorem for the weak convergence of probability measures on the space of tempered distributions (see Theorem 2.3 in \cite{bierme2017generalized}). According to this, in order to prove weak convergence of $X^{H}$ towards a LGF as $H$ tends to $0$, we need to show that for any $\phi_{1},\phi_{2}\in \mathcal S$
$$\lim_{H\to0}E[( X^{H}, \phi_{1} )( X^{H}, \phi_{2} )] =  \int_{\Rd}\int_{\Rd} K(x,y) \phi_{1}(x)\phi_{2}(y)\,dx\,dy, 
$$  
where 
$$
K(x,y)=\log\frac{1}{\|x-y\|} +g(x,y). 
$$

This was done for the one dimensional case and for a specific normalizing kernel $\psi$ in Theorem 2.1 of \cite{neuman2018fractional}.
Since the focus of this work is the convergence of the multiplicative chaos associated to $X^{H}$ we do not pursue this direction here.
\end{remark}

Next we give a few examples of normalizing kernels in $\mathcal N_{H_0}(D)$. We are mainly interested in normalizations that preserve one of the two characterizing properties of fractional Brownian fields: stationarity (and isotropy) of increments and self-similarity.

\textbf{Stationarity of increments:} we choose the normalizing kernel $\psi$ to be a convolution kernel. Let $\theta: \Rd \to \R_+$ be a measurable bounded function such that
\begin{align*}
\psi(x,y) = \theta(y-x), \quad\text{with}\quad\int_{\Rd}\theta(x)dx = 1.
\end{align*}
In this case the conditions in Definition \ref{def-ker} translate to conditions on $\theta$.

It is straight forward to check that the conditions of Definition \ref{def-ker} are satisfied for any positive $\theta \in \mathcal{S}$ with the domain $D$ being any bounded subset of $\Rd$.
Another interesting example is given by $\theta(x) = \abs{A}^{-1}\mathbf{1}_{A}(x)$
for any bounded set $A \subset \Rd$ where $D$ is again a bounded domain. Here $\abs{A}$ is the Lebesgue measure of the set $A$.
It is a simple exercise to show that in these examples $X^{H}$ in \eqref{def:XH} inherits the stationarity of the increments from $B^H$. Further, if $\theta$ is invariant under rotations, it is also straight forward to show that $X^{H}$ has isotropic increments.

A specific example that preservers stationarity of increments in $d=1$ is the following moving average normalization
\begin{align*}
B^{H}(t) - \frac{1}{\delta}\int_{t-\delta}^{t} B^{H}(u) du
\end{align*}
for any fixed $\delta>0$. Here the normalized process is adapted to the filtration generated by the fractional Brownian motion.
Note that this type of normalization does not preserve the self-similarity of $B^H$, as we discuss next.  

\textbf{Self-similarity:} using the self similarity property of the FBF, we get for any  $\psi \in \mathcal{N}_{H_0}(D)$ and $A \subset D$, $\lambda > 0$ with $\lambda A \subset D$ that 
\begin{align*}
    \left(B^{H}(\lambda x) - \!\int_{\Rd} B^{H}(u)\psi(u, \lambda x)du\right)_{x\in A} \!\overset{d}{=} \lambda^{H}\!\!\left( B^{H}(x) - \lambda\! 
    \int_{\Rd} B^{H}(u)\psi(\lambda u, \lambda x) du \right)_{x\in A}.
\end{align*}
By imposing $\psi(\lambda x, \lambda y) = \lambda^{-1} \psi(x, y)$, we see that the normalizations $X^H$ in \eqref{def:XH} preserve the self-similarity.
In particular, kernels of the form $\psi(x,y) = \norm{y}^{-1}\theta(x\norm{y}^{-1})$ satisfy this property, where the conditions on $\psi$ translate to conditions on ${\theta: \Rd \to \R_+}$. 
Note that the previous examples: $\theta \in \mathcal{S}$ and $\theta = |A|^{-1}\mathbf{1}_A$,
with $A\subset\Rd$ being bounded, apply also in this case, if the domain $D$ is any bounded subset of $\Rd$ that excludes a neighbourhood of zero.

A special case of the preceding example is the normalization proposed in \cite{neuman2018fractional} which is given in \eqref{nr-norm}. In this case $ D := [\delta, T]$ for some $0<\delta < T$ and  
\be \label{ker-nr}
\psi(t,u) = \frac{1}{t}\mathbf{1}_{[0,t]}(u),
\ee
Hence this normalization preserves self-similarity and also keeps $X^H$ adapted to the filtration generated by the fractional Brownian motion.
See Figure \ref{fig:simulations_nr_nrom} for realisations of the normalized Gaussian process with $\psi$ as in \eqref{ker-nr}. 

\begin{figure}
    \centering
    \includegraphics[width=0.8\textwidth]{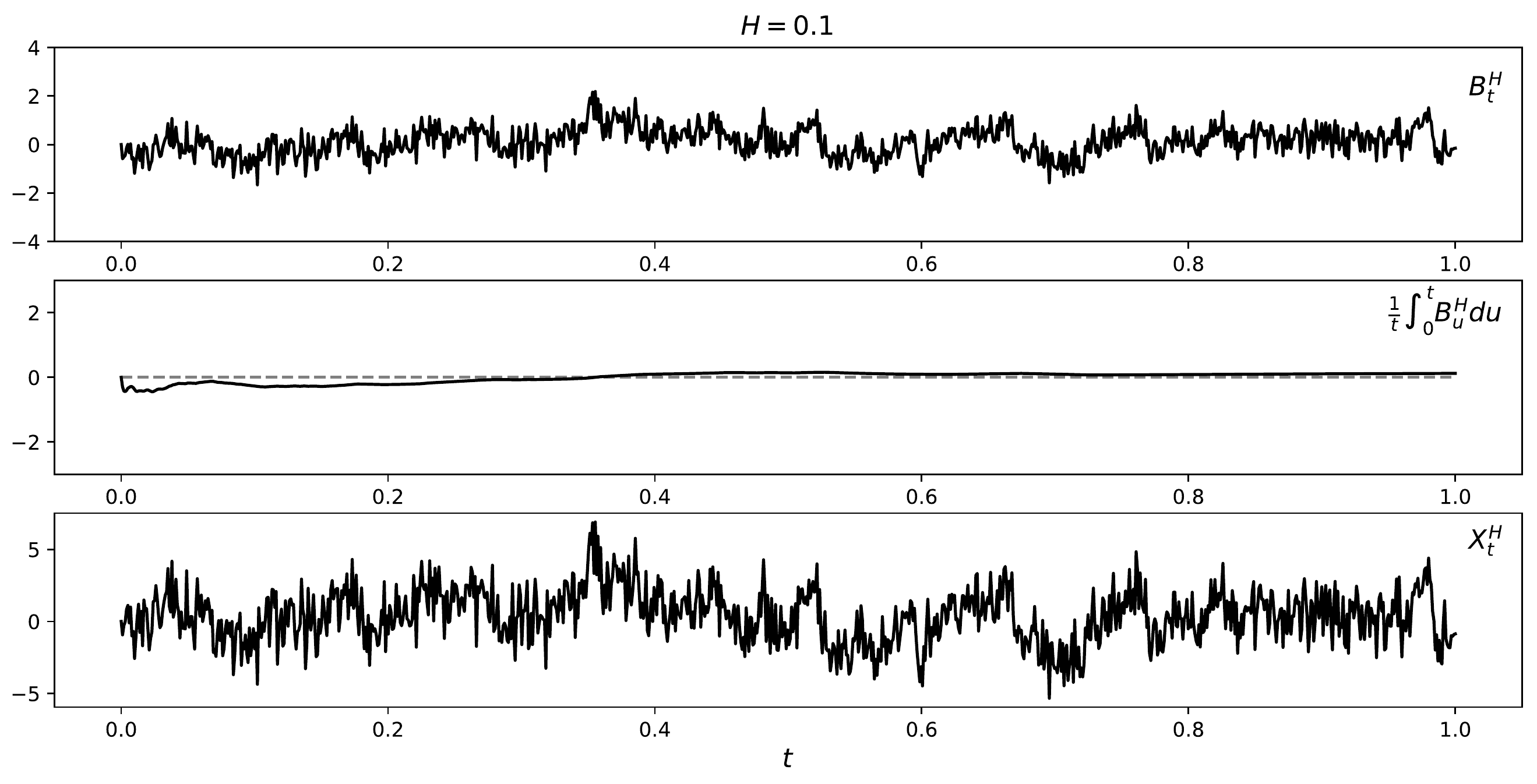}
    \includegraphics[width=0.8\textwidth]{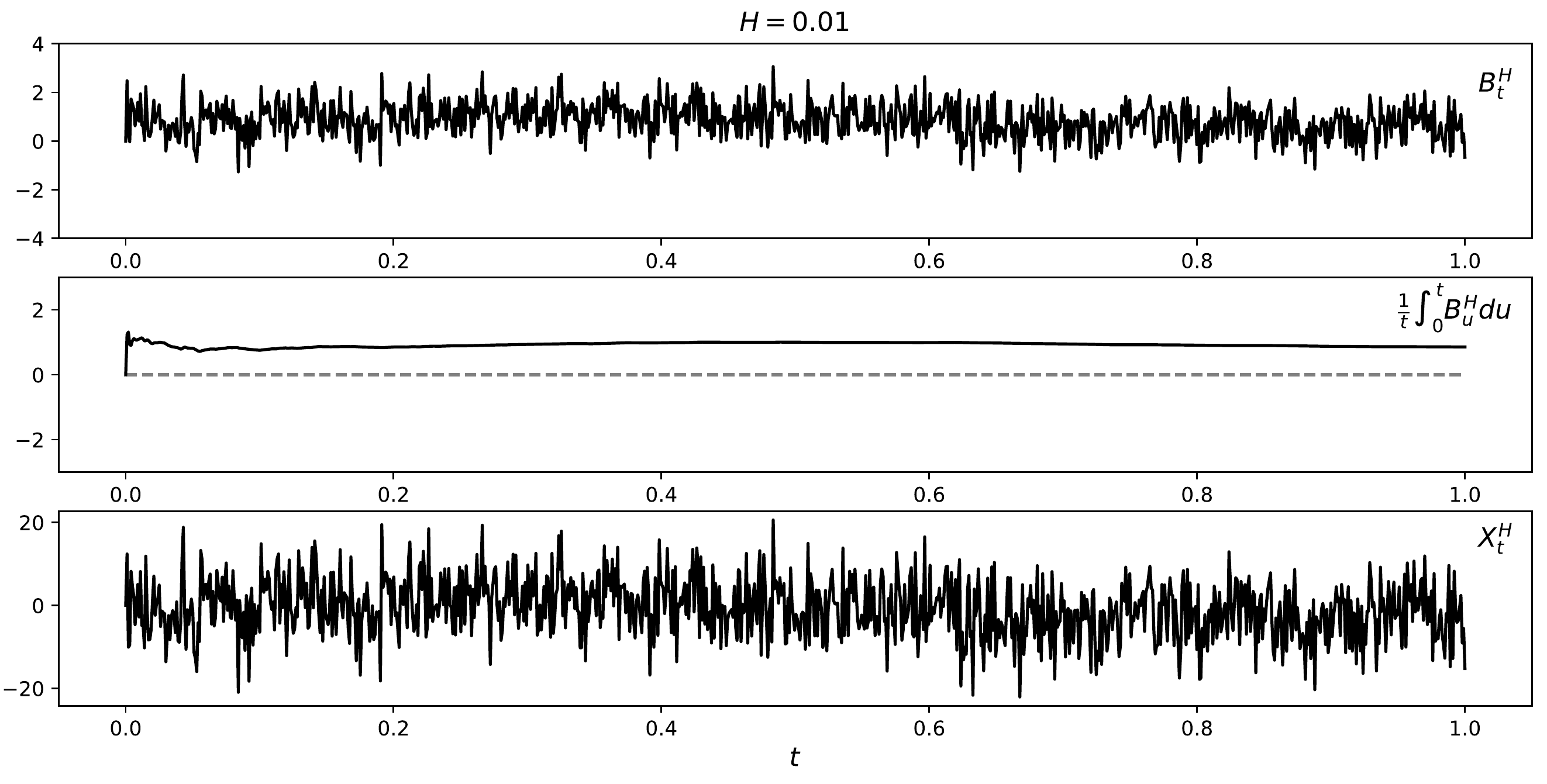}
    \caption{Two realisations of the normalized fractional Brownian motion \eqref{nr-norm} for Hurst parameters $H=0.1$ (upper panel) and $H=0.01$ (lower panel). In each panel the first realisation is of fractional Brownian motion, the second is the subtracted integral term and the third is the normalized fractional Brownian motion.}
    \label{fig:simulations_nr_nrom}
\end{figure}

Another example for such kernels is given by $\theta(x) = \abs{B_1(0)}^{-1} \mathbf{1}_{B_1(0)}(x)$ where $B_1(0)$ is the unit ball around $0$ and $D$ is any bounded domain in $\Rd$ that excludes a neighbourhood of zero. 

Finally, let us give an intuitive argument that explains why a normalization of FBFs is needed in order to establish convergence when $H$ tends to $0$. As we discusses in Section \ref{sec:introduction}, FBFs and LGFs are subsets of the class of fractional Gaussian fields $\{\textrm{FGF}_s(\re^d), \, s\in \re\}$.
For $H = s - d/2 \in [0,1)$, the distributions of these fields give full measure to (a representation of) the quotient space $\mathcal{S}_0^{'} = \mathcal{S}^{'}/\mathcal{S}_0$, where $\mathcal{S}_0\subset \mathcal S$ is the sub-space functions that integrate to zero.
In other words, these fields are defined as random tempered distributions \emph{modulo a constant}.
The convergence of the fields as $H \downarrow 0$ is a phase transition in the FGF-class.
For $H \in (0,1)$ the samples of $\mathrm{FGF}_s(\re^d)$ are tempered distributions $h \in \mathcal{S}_0^{'}$ which admit representations as continuous functions.
Fixing the undefined constants of these distributions by requiring an evaluation zero at the origin, i.e. requiring $\langle h, \delta_0\rangle = 0$, where $\delta_0$ is the Dirac distribution, gives up to re-scaling by a constant, the FBF with Hurst parameter $H$.
However, for $H=0$ the samples of $h\in\mathcal{S}^{'}_0$ of the log-correlated field $\mathrm{FGF}_0(\re^d)$ are not representable by continuous functions and testing $h$ against $\delta_0$ is not possible.
Therefore, requiring the FGF to be zero at the origin leads to a condition that is ill-defined in the $H\downarrow 0$ limit.
In order to obtain a meaningful limit, one has to loosen the latter condition in such a way that it can also be imposed on the LGF. This is precisely what the class for normalizations in Theorem \ref{thm:fbm_normalization} does in a general form.

Without going further into detail, a modification of the conditions on class normalizing kernels, to unbounded domains is possible.
We have seen however, that in order to obtain self-similarity, the domain of the normalized field has to exclude a neighbourhood of zero, which is clearly breaking the scale invariance of the domain and therefore the global self-similarity.
Intuitively this is explained by noting that the global self-similarity property in the $H\downarrow0$ limit corresponds to the scale-invariance of the field, which is indeed a characteristic property of the log-correlated field, however it only makes sense when understanding the field modulo constants.

\paragraph{Organisation of the paper:} 
The rest of this paper is dedicated to the proofs of the main results in Theorems \ref{thm-convergence} and \ref{thm:fbm_normalization}. 
In Section \ref{sec-ui} we prove uniform integrability for the family of measures $\{M^{H}_\gamma\}_{H \in (0, H_0)}$ from Theorem \ref{thm-convergence}. In Sections \ref{sec-pr-prop35} and \ref{sec-lem-p} we prove  Proposition \ref{prop-j1} and  Lemma \ref{lem-p-star}, respectively, which are essential ingredients for the proof of uniform integrability. 
In section \ref{section-conv} we use uniform integrability in order to prove the convergence of  $\{M^{H}_\gamma\}_{H \in (0, H_0)}$ as $H\downarrow 0$. 
Section \ref{corr-comp} is dedicated to the proof of Lemma \ref{lem:proof_fbf_covariance}. Finally  in Section \ref{sec:proof_normalization} we prove Theorem \ref{thm:fbm_normalization}.

\section{Uniform Integrability} \label{sec-ui}
In this section we show that the family of measures $\{M^{H}_\gamma\}_{H\in (0,H_0)}$ from  (\ref{approxMeasFBM}) are uniformly integrable.  The result is given in the following proposition. 

Let $\mathcal A$ be the class measurable subsets of $D$.

\begin{proposition} \label{prop:uniform_integrability}
For any $A\in \mathcal A$, $\big \{M^{H}_\gamma(A)\big\}_{H\in(0, H_0)}$ is uniformly integrable on $(\Omega, \F, P)$, for all $\gamma \leq \gamma^{*}(d)$, with $\gamma^{*}(d)>\sqrt{\frac{7}{4}d}$.
\end{proposition}

\noindent The main idea in the proof of Proposition \ref{prop:uniform_integrability}, is to restrict the limiting measure to so called 
\emph{good points}, that is points $x \in A$ in which the field does not deviate too much from its mean. To be more precise, let $\ol H \in (0,H_0)$. We define the event of $x$ being a good point of order $\alpha>0$ by
\begin{align} \label{g-set} 
G^{H,\ol H}_\alpha(x) &= \left\{ X^{h}(x) \;\le\; \frac{\alpha}{h+H}, \foralll h \in S_{H,\ol H} \right\},
\end{align}
where $H \in(0,\ol H/2)$ and we define the following grid of $h$'s by
\be \label{s-set} 
S_{H, \ol H} = \left\{h: h = H + \frac{1}{n}, \; n\in\mathbb{N},\; \frac{1}{\ol H - H} < n \le \frac{1}{H}\right\}. 
 \ee
 
Before we prove Proposition \ref{prop:uniform_integrability}, we introduce a sequence of auxiliary lemmas. 
The following two lemmas will motivate the restriction of the random field $X^H$ to the good points.  

\begin{lemma} \label{ordinarilyGood} For any $\alpha > 0$ and $\ol H \in (0, H_0)$, there exists $p^{\ol H}_\alpha > 0$ such that 
$$
P\big(G^{H,\ol H}_\alpha(x)\big) \ge 1 - p^{\ol H}_\alpha, \quad   \foralll x\in D,  \ 0< H \le \frac{1}{2}\ol H.
$$ 
Moreover, $p^{\ol H}_\alpha \to 0$ as $\ol H \to 0$.
\end{lemma}

\begin{proof}
We will bound the probability of the event $G_\alpha^{H, \ol H}(x)$ from below by bounding the probability of the complementary event $G^{H,\ol H}_\alpha(x)^c$ from above as follows
\be \label{eq-r1} 
\begin{aligned}
P\big(G^{H,\ol H}_\alpha(x)^c \big) 
& = P\left(\exists \, h \in S_{H, \ol H}, \ \text{ s.t. }\; X^{h}(x) > \frac{\alpha}{h+H} \right) \\
&\less \sum_{h \in S_{H, \ol H}}P\left(X^{h}(x) > \frac{\alpha}{h+H}\right) .
\end{aligned}
\ee
From (\ref{eq:XHCovariance})--(\ref{eq:gEstimate}) we get
\be \label{cov-bnd} 
0 < E(X^{h}(x)^{2}) \le (1+c_1) \Big( \frac{1}{2h} + c_2\Big), \quad \textrm{ for every } x\in D, \, h\in(0,H_0), 
\ee
for some constants $c_1,c_2 > 0$.  

We will use the following tail estimate for a random variable $Z$, which is a centred Gaussian with variance $\sigma^2$, \be \label{z-est} 
P(Z >x) \leq e^{-\frac{x^{2}}{2\sigma^2}}, \quad \textrm{for all } x>0. 
\ee
Let $h=H+1/n$ as in \eqref{s-set}. Use \eqref{cov-bnd} and  (\ref{z-est}) to get
\be \label{peq1} 
\begin{aligned}
P\left(X^{h}(x) > \frac{\alpha}{h+H} \right) 
\le & \exp\left(\!\!-\frac{\alpha^{2}\big(2H+\frac{1}{n}\big)^{-2}}{2  \E\big( X^{H + 1/n}(x)^{2} \big)}\!\right) \\
\le & \exp\left(\!\!-\frac{\alpha^{2}}{2(c_1 + 1)}\frac{\big(2H+\frac{1}{n}\big)^{-2}}{\big(H+\frac{1}{n}\big)^{-1} + c_2}\right). 
\end{aligned}
\ee
Note that 
\bn
\frac{\big(2H+1/n\big)^{-2}}{\big(H+1/n\big)^{-1} + c_2}&=& \frac{1}{(2H+1/n)^{2}} \frac{H+1/n}{1+c_2(H+1/n)} \\
&=&  \frac{1}{2(2H+1/n)^{2}} \frac{2H+2/n}{1+c_2(H+1/n)} \\
&\ge&   \frac{1}{2(2H+1/n)} \frac{1}{1+c_2(H+1/n)} \\ 
&\geq&   \frac{1}{4(H+1/n)} \frac{1}{1+c_2},
\en
where we used \eqref{s-set}  and the fact that $H+1/n \leq \ol H  \leq H_0 < 1$ in the last line. 

For any $x>0$ define $[x]$  to be the largest integer less than or equal to $x$. Denote $m=\left[\frac{1}{H}\right]$. Since $m\leq \frac{1}{H}$ we get
\bd
  H+\frac{1}{n} \leq \frac{1}{m} +\frac{1}{n}, 
\ed
and it follows that 
\be \label{peq2} 
\frac{\big(2H+1/n\big)^{-2}}{\big(H+1/n\big)^{-1} + c_2}  \geq   \frac{1}{4(\frac{1}{m}+\frac{1}{n})} \frac{1}{1+c_2}. 
\ee
From \eqref{peq1} and \eqref{peq2} we get that
\be \label{r12} 
P\left(X^{h}(x) > \frac{\alpha}{h+H} \right) \leq  \exp\left(-\frac{\alpha^{2}}{8(c_2+1)(c_1 + 1)}\frac{1}{\frac{1}{m}+\frac{1}{n}}\right).
\ee
Define $\beta = \frac{\alpha^{2}}{8(c_1+1)(c_2 +1)} >0$. Then from \eqref{eq-r1} and (\ref{r12}) we get that  
\begin{align*}
P\big(G^{H,\ol H}_\alpha(x)^c \big) &\le \sum_{n=[1/\ol H] -1}^{m}  \exp\left(-\frac{\beta}{\frac{1}{m}+\frac{1}{n}}\right) \\
&= \exp\left(-\frac{\beta}{2}m\right)+ \sum_{n=[1/\ol H] -1}^{m-1}  \exp\left(-\frac{\beta}{\frac{1}{m}+\frac{1}{n}}\right) \\
&\leq \exp\left(-\frac{\beta}{2}m\right) +\sum_{n=[1/\ol H] -1}^{m-1}  \exp\left(-\frac{\beta}{\frac{1}{m-1}+\frac{1}{n}}\right). 
\end{align*}
By iterating the preceding inequality we get
\begin{align*}
P\big(G^{H,\ol H}_\alpha(x)^c \big) 
&\le \sum_{n=[1/\ol H] -1}^{m} \exp\left(-\frac{\beta}{2}n\right) \\
&< \sum_{n=[1/\ol H] -1}^{\infty} \exp\left(-\frac{\beta}{2}n\right) =: p^{\ol H}_\alpha.
\end{align*}
It is clear then that $p^{\ol H}_\alpha \rr 0$ as $\ol H \rr 0$. 
\end{proof}

Next, introduce additional  definitions that will be used throughout this section. 

For a centred Gaussian random variable $\xi$ we define, 
$$
 \overline{\xi}=  \xi -\frac{1}{2}E(\xi^2).
$$
Further we define the measure $\wt {P}$ by 
\be \label{lr} 
\frac{d\wt P}{dP} = e^{\gamma\overline{ X^{H}}(x)}.
\ee

\begin{lemma}  \label{le:GoodPointsEnough} Let $\alpha > \gamma$. For any $\eps  \in (0,\alpha/\gamma - 1) $ there exists $\ol H>0$ sufficiently small such that 
$$
E\big[ e^{\gamma\overline{ X^{H}}(x)} \mathbf{1}_{G^{H,\ol H}_\alpha(x)} \big] \geq 1- p^{\ol H}_{\alpha-\gamma(1+\eps)}, \quad   \foralll x\in D,  \ 0< H \le \frac{1}{2}\ol H, 
$$
where $p^{\ol H}_{\alpha-\gamma(1+\eps)}$ is given in Lemma \ref{ordinarilyGood}.
\end{lemma}

\begin{proof}
By Cameron-Martin-Girsanov theorem under the measure $\wt{P}$, the Gaussian process $(X^{h}(x))_{h\in(0,1)}$ has similar variance as under $P$ and a shifted mean which is bounded by
\be \label{cov-h-H} 
\gamma E\big[X^{h}(x)X^{H}(x)\big]\leq \gamma C_{H, h}\left( \frac{1}{H+h} +c\right), \quad \textrm{for all } x\in D, 
\ee
where $c>0$ is a constant independent of $x$ and $H$. Note that we have used \eqref{eq:XHCovariance} and \eqref{eq:gEstimate} in the above inequality. 
 
Let $0<\eps <\alpha/\gamma - 1$.  Recall that for $h\in S_{H, \ol H} $ we have $2H \leq h \leq \ol H$, then using (\ref{c-h}) we get that for $\ol H$ small enough that
$$
\gamma E\big[X^{h}(x)X^{H}(x)\big] \leq \gamma (1+\eps) \frac{1}{h+H}, \quad \textrm{for all } x\in D,  \ h\in S_{H, \ol H}.
$$
  
Using the above inequality, \eqref{g-set} and \eqref{lr} we get that
\begin{align*}
&E\big[e^{\gamma\overline{ X^{H}}(x)} \mathbf{1}_{\{G^{H,\ol H}_\alpha(x)\}} \big]  \\
&=\wt P\big(G^{H,\ol H}_\alpha(x)\big) \\
&= P\Big(X^{h}(x) \leq \frac{\alpha}{h+H}  -\gamma \E\big(X^{h}(x)X^{H}(x)\big) ,\ \forall \, h \in  S_{H, \ol H}  \Big) \\
& \geq  P\Big(X^{h}(x) \;\le\; \alpha\,\frac{1}{h+H} - \gamma (1+\eps)\frac{1}{h+H}  ,\ \forall \, h \in  S_{H, \ol H} \Big) \\
&= P\big(G^{H,\ol H}_{\alpha - \gamma(1+\eps)}(x)\big).
\end{align*}
We can thus conclude the result from Lemma \ref{ordinarilyGood}.
\end{proof} 

Recall that $\mathcal A$ is the class measurable subsets of $D$. For any $0<\ol H<H_0$, $H\in(0,\ol H/2)$ and $\alpha > \gamma$ we define the random measure $I^{H,\ol H}_{\alpha,\gamma}(\cdot) $ as follows
\begin{align}\label{eq:def_I}
I^{H,\ol H}_{\alpha,\gamma}(A)=  \int_A  e^{\gamma X^{H}(x) - \frac{\gamma^{2}}{2}E[X^{H}(x)^{2}] } \mathbf{1}_{G^{H, \ol H}_\alpha(x)} dx, \quad A\in \mathcal A. 
\end{align}
Note that $I^{H,\ol H}_{\alpha,\gamma}(\cdot)$ is the approximating measure $M^{H}_\gamma$ in \eqref{approxMeasFBM}, restricted to good points. In the following proposition we derive a uniform bound on the second moment of $I^{H,\ol H}_{\alpha,\gamma}$.

\begin{proposition}\label{prop:I_L2_boundedness} There exists $\gamma^*(d) > \sqrt{\frac{7}{4}d}$ such that for any $\gamma \in (0,\gamma^*(d))$ and $\alpha > \gamma$ sufficiently close to $\gamma$, there exists $\ol H >0$ sufficiently small such that 
$$
\sup_{0<H \leq \ol H/2}\sup_{A\in \mathcal A}E\big[I^{H,\ol H}_{\alpha,\gamma}(A)^{2} \big] < \infty. 
$$
\end{proposition}
\begin{proof}
Let $A \in \mathcal A$. For any $H\in(0,H_0)$ define the probability measure $\ol P$ by 
\be \label{ol-p} 
\frac{d \ol {P}}{dP}=e^{\gamma X^{H}(x) + \gamma X^{H}(y) - \frac{\gamma^{2}}{2}E[X^{H}(x) + X^{H}(y)]^{2}}.
\ee
From (\ref{eq:XHCovariance}), (\ref{eq:gEstimate}), \eqref{ol-p} and Fubini's theorem we get
\be
\begin{aligned}\label{IHL2Bound}
&E\left[I^{H,\ol H}_{\alpha,\gamma}(A)^{2} \right]\\
&= \int_A\int_A E\left[ e^{\gamma X^{H}(x) + \gamma X^{H}(y) - \frac{\gamma^{2}}{2}E[X^{H}(x)^{2}] - \frac{\gamma^{2}}{2}E[X^{H}(y)^{2}]} \mathbf{1}_{\{G^{H}_\alpha(x) \cap G^{H}_\alpha(y)\}} \right] dx dy \\
&= \int_A\int_A e^{\gamma^{2}E\left[X^{H}(x)X^{H}(y)\right]}\; \ol{P}(G^{H}_\alpha(x) \cap G^{H}_\alpha(y)) dx dy \\
&= \int_A\int_A \exp\Big(C_{H,H}\gamma^{2}\frac{1-\norm{x-y}^{2H}}{2H} + g^{2H}(x,y)\Big)\,  \ol P(G^{H}_\alpha(x) \cap G^{H}_\alpha(y)) dx dy \\
&\leq K \int_A\int_A \exp\Big(C_{H,H} \gamma^{2}\frac{1-\norm{x-y}^{2H}}{2H}\Big)\,  \ol P(G^{H}_\alpha(x) \cap G^{H}_\alpha(y)) dx dy.
\end{aligned}
\ee
Our goal is to bound $E\big[I^{H,\ol H}_{\alpha,\gamma}(A)^{2} \big]$ uniformly in $H$ and $A$. In order to do that, we split the integral on the right-hand side of (\ref{IHL2Bound}) to four regions. Let $\kappa^{*} \geq 1$ be a constant that will be fixed later. 

Define: 
\begin{equation} \label{region} 
\begin{aligned} 
R_1:= &\{(x,y) \in D \times D: \norm{x-y}  < e^{-\kappa^{*}/H}\}; \\
R_2:= &\{(x,y) \in D \times D: e^{-\kappa^{*}/H} \leq \norm{x-y}  < e^{-2/\overline{H}} \};
\\ 
R_3:=&\{(x,y) \in D \times D: e^{-2/\overline{H}} \le  \norm{x-y} <  1\}; \\
R_4:=&\{(x,y) \in D \times D:  1\le  \norm{x-y} \big \}. 
\end{aligned}
\end{equation} 
Note that since $H/\ol H \in (0,1/2)$ and $\kappa^{*} \geq 1$, $R_i$, $i=0,1,2,3$ are disjoint and non-empty.

We further define 
\be \label{j-i}
J_i(H,\overline{H},A)  = \iint_{R_i} \exp\Big(C_{H,H} \gamma^{2}\frac{1-\norm{x-y}^{2H}}{2H}\Big)\,  \ol P(G^{H}_\alpha(x) \cap G^{H}_\alpha(y)) dx dy.
\ee
Note that $J_i(H,\overline{H},A)$ depend also on $\gamma$ and $\alpha$. We suppress this dependence in order to simplify the notation. 
 
From (\ref{IHL2Bound}) it follows that
\be \label{i-b}
 E\big[I^{H,\ol H}_{\alpha,\gamma}(A)^{2}\big] \leq K \sum_{i=1}^{4}J_i(H,\overline{H},A) .
\ee
Our next goal will be to bound $J_i$, $i=1,...,4$. 

Using \eqref{c-h} we notice that for any for arbitrarily small $\dl>0$ we can choose $\ol H$ small enough, such that for all $H\in (0,\ol H)$, 
\be \label{j-1-b} 
\begin{aligned} 
J_1(H,\overline{H},A)  &\leq \iint_{R_1} \exp\Big(C_{H,H} \gamma^{2}\frac{1-\norm{x-y}^{2H}}{2H}\Big)dx dy \\ 
&\leq \iint_{R_1} \exp\Big((1+\dl) \gamma^{2}\frac{1}{2H}\Big)dx dy \\ 
&\leq \exp\Big((1+\dl) \gamma^{2}\frac{1}{2H}\Big) |R_{1}| \\ 
&\leq C \exp\Big((1+\dl) \gamma^{2}\frac{1}{2H}-\frac{\kappa^{*}d}{H}\Big). 
\end{aligned} 
\ee
Since $\kappa^{*}\geq 1$ and $\dl$ is arbitrarily close to $1$, it follows that for $\ol H$ sufficiently small,  
\be \label{J1-b}
\sup_{H\leq \ol H/2}\sup_{A\in \mathcal A}J_1(H,\overline{H},A) <\infty , \quad \textrm{for all }  \gamma^{2} < 2d,  
\ee
as needed.

Since $A \subset D$ and $D$ is a bounded domain, the following bound on $J_4$ follows trivially, 
\be \label{J3-b} 
\sup_{0< H\leq \ol H/2}\sup_{A\in \mathcal A}J_4(H,\overline{H},A) <\infty. 
\ee
Next, we use the inequality 
$$
\frac{1-\norm{x-y}^{2H}}{2H} \leq - \log\norm{x-y}, \quad \textrm{for } \norm{x-y} < 1, \ H\in (0,1/2), 
$$
together with \eqref{region} and \eqref{j-i} we get, 
\begin{align*}
J_3(H, \overline{H},A) &\le \iint_{R_3} \exp\Big(\gamma^{2}\frac{1-\norm{x-y}^{2H}}{2H}\Big)  dx dy \\
&\le \iint_{ R_3} \exp\Big(\gamma^{2}(-\log\norm{x-y})\Big)  dx dy  \\
&\le \int_{D}\int_{D} e^{\gamma^{2} 2\overline{H}^{-1}}dxdy  \\
&\le e^{2\gamma^{2}\overline{H}^{-1}}|D|^{2}.
\end{align*}
It follows that 
\be \label{J2-b} 
\sup_{0< H\leq \ol H/2}\sup_{A\in \mathcal A} J_3(H,\overline{H},A) <\infty. 
\ee
The derivation of a uniform bound on $J_{2}(H, \overline{H},A)$ is long and involved. Therefore, we summarise the result in the following Proposition, which will be proved in Section \ref{sec-pr-prop35}. 
\begin{proposition}  \label{prop-j1}
There exists $\kappa^*>1$, $\gamma^{*}(d) >  \sqrt{\frac{7}{4}d}$ and $\ol H \in (0,H_0)$, such that for all $\gamma \leq \gamma^{*}(d)$ we have
$$
\sup_{0< H\leq \ol H/2}\sup_{A\in \mathcal A}J_{2}(H, \overline{H},a) <\infty. 
$$
\end{proposition} 

From (\ref{i-b}), (\ref{J1-b}) (\ref{J3-b}), (\ref{J2-b}) and Proposition \ref{prop-j1} we get the result of Proposition \ref{prop:I_L2_boundedness}. 
\end{proof}

The result of Proposition \ref{prop:I_L2_boundedness} is the main ingredient in the proof of uniform integrability of $\big \{M^{H}_\gamma(S)\big\}_{H\in(0,\ol H/2 )}$, as shown later in the proof of Proposition \ref{prop:uniform_integrability}.

Before we present the proof of Proposition \ref{prop:uniform_integrability}
we will state the following useful  corollary. 

For any $\alpha >\gamma$, $\ol H\in (0,H_0)$, and $H\in (0,\ol H/2)$ we define 
\be\label{eq:def_L}
L^{H,\ol H}_{\alpha,\gamma}(A)=  \int_A  e^{\gamma X^{H}(x) - \frac{\gamma^{2}}{2} E[X^{H}(x)^{2}] } \mathbf{1}_{\big(G^{H, \ol H}_\alpha(x)\big)^c} dx, \quad A \in \mathcal A. 
\ee
\begin{corollary} \label{corr-supp}
Let $A\in \mathcal A$ and $\alpha > \gamma$. Then for every $\eps>0$, there exists $\ol H>0$ sufficiently small such that 
\begin{align*} 
\sup_{0<H \leq \ol H/2}  \E \big[ L^{H, \ol H}_{\alpha,\gamma} (A)\big] &\leq \eps.
\end{align*} 
\end{corollary} 
\begin{proof}
Fix $A\in \mathcal A$ and $\alpha >\gamma$. Let $\eps>0$. From Lemma \ref{le:GoodPointsEnough} we can choose $ \bar \eps \in (0,\alpha/\gamma-1)$ and $\ol H$ small enough such that 
\bn
\sup_{0<H \leq \ol H/2}  \E \big[ L^{\ol H}_{\alpha,\gamma} (A)\big] &\leq& |A|p^{\ol H}_{\alpha-\gamma(1+ \bar \eps)} \\
&\leq& \eps. 
\en
\end{proof}

\begin{remark} \label{rem-supp}  
Note that Corollary \ref{corr-supp} provides some valuable information on the support of the measure $M^H_\gamma(\cdot)$, which were defined in \eqref{approxMeasFBM}. Indeed we show that for $H$ small enough, $M^H_\gamma(\cdot)$ is arbitrarily small outside the set of good points $\cup_{x\in D}G^{H, \ol H}_\alpha(x)$, with probability close to $1$. 
\end{remark} 
Now we have all the ingredients for the proof of Proposition \ref{prop:uniform_integrability}. 

\begin{proof}[Proof of Proposition \ref{prop:uniform_integrability}]  

Recall that $M^{H}_\gamma$, $I^{H, \ol H}_{\alpha,\gamma}$ and $L^{H, \ol H}_{\alpha,\gamma}$ were defined in \eqref{approxMeasFBM}, (\ref{eq:def_I}) and \eqref{eq:def_L}, receptively. We therefore have,
\be \label{m-dec}
 M^{H}_\gamma(A)  =  I^{H, \ol H}_{\alpha,\gamma} (A) +  L^{H, \ol H}_{\alpha,\gamma}(A)  .
\ee
for any $\ol H \in (0,H_0)$. From Proposition \ref{prop:I_L2_boundedness}, we get that there exists $\alpha > \gamma$ sufficiently close to $\gamma$, such that for all $\ol H >0$ sufficiently small we have
\be\label{fd1} 
\sup_{0<H \leq \ol H/2}\sup_{A\in \mathcal A}E\big[I^{H,\ol H}_{\alpha,\gamma}(A)^{2} \big] < \infty.
\ee
It follows that $\{I^{H,\ol H}_{\alpha,\gamma}(A)\}_{0<H \leq \ol H/2}$ are uniformly integrable. 

Let $ \eps > 0$ be arbitrarily small and choose $B\in \mathcal F$ such that 
\be \label{unif-i}
\sup_{0<H \leq \ol H/2} E\big[I^{H,\ol H}_{\alpha,\gamma}(A) \mathbf{1}_{B}\big] < \frac{\eps}{2}.
\ee
From Corollary \ref{corr-supp} with $\alpha$ which was fixed in \eqref{fd1}, we have for $\ol H$ small enough 
\be \label{ggh1} 
\sup_{0<H \leq \ol H/2}  \E \big[ L^{H, \ol H}_{\alpha,\gamma} (A)\big] \leq \frac{ \eps}{2}.
\ee

From \eqref{m-dec}, \eqref{unif-i} and  \eqref{ggh1} it follows that there exists $\ol H$ small enough such that 
$$
\sup_{0<H<\ol H/2}E\big[M^{H}_\gamma(A) \mathbf{1}_{B}\big] \leq \eps,
$$
and the uniform integrability of $\big \{M^{H}_\gamma(A)\big\}_{H\in(0,\ol H)}$ follows.  

Next we will show that $\big \{M^{H}_\gamma(A)\big\}_{H\in [ \ol H,H_0)}$ is abounded in $L^{2}$, this will conclude the proof. Repeating the same steps as in \eqref{IHL2Bound} and then using \eqref{c-h} and \eqref{eq:gEstimate} we get that there exist constants  $C_{1}, C_{2}>0$ such that 
\begin{align*} 
E\left[ M_{\gamma}^{H}(A)^{2} \right] &= \int_A\int_A \exp\Big(C_{H,H}\gamma^{2}\frac{1-\norm{x-y}^{2H}}{2H} + g^{2H}(x,y)\Big) dx dy \\
&\leq  \int_A\int_A \exp\Big((1+C_{1})\gamma^{2}\frac{1-\norm{x-y}^{2H}}{2H} + C_{2}\Big) dx dy \\
&\leq  \int \int_{\|x-y\| <1} \exp\Big((1+C_{1})\gamma^{2}\frac{1-\norm{x-y}^{2\ol H}}{2\ol H} + C_{2}\Big) dx dy \\
&\quad +  \int \int_{\|x-y\| \geq 1} \exp\Big((1+C_{1})\gamma^{2}\frac{1-\norm{x-y}^{2H}}{2H} + C_{2}\Big) dx dy \\
&\leq  \wt C_{1}(\ol H)+ \wt C_{2}|A|^{2}, \quad \textrm{for all }  \ol H \leq H \leq H_0. 
 \end{align*} 
It follows that 
$$
\sup_{\ol H \leq H < H_0}E\left[ M_{\gamma}^{H}(A)^{2} \right] <\infty. 
$$

 \end{proof}

 \section{Proof of Proposition \ref{prop-j1}} \label{sec-pr-prop35} 
In order to get a uniform bound on $J_{2}$ in \eqref{j-i} we first need to bound $\ol P\big({G^{H}_\alpha(x)\cap G^{H}_\alpha(y)}\big)$. 

Using (\ref{eq:XHCovariance}) it follows that the Cameron-Martin shift due to the change of measure \eqref{ol-p} is given by  
\begin{align*}
&\gamma E\big[ X^{h}(x) \left(X^{H}(x) + X^{H}(y)\right) \big] \\
&=\gamma C_{h,H}\left( \frac{1}{H + h} +  \frac{1 - \norm{x-y}^{H+h}}{H+h} + g^{H,h}(x,x) +g^{H,h}(x,y) \right).
\end{align*}
Let $\dl>0$ be arbitrarily small. From (\ref{c-h}) we get that for all sufficiently small $\ol H$ we have 
\be\label{c-h-bnd} 
\sup_{0<h,H < \ol H} |C_{h,H}| \geq 1-\dl.
\ee 
From \eqref{eq:gEstimate} it follows that there exists a constant $C_{1}>0$ such that for all sufficiently small $\ol H$ we have 
\bd   
\sup_{0<h,H < \ol H}\sup_{x,y \in D}|g_{h,H}(x,y)| \leq C_{1}. 
\ed 
Together with \eqref{g-set} it follows that by choosing $\ol H$ small enough, for all $0<H < \ol H/2$, $h\in S_{H, \ol H} $ and $x,y \in D$ we have   
\be \label{drfit101}
\begin{aligned}
\gamma E\big[ X^{h}(x) \left(X^{H}(x) + X^{H}(y)\right) \big] 
&\geq \gamma (1-\dl) \left( \frac{1}{H + h} +  \frac{1 - \norm{x-y}^{H+h}}{H+h} -2C_{1} \right) \\
&\geq \gamma (1-2\dl) \left( \frac{1}{H + h} +  \frac{1 - \norm{x-y}^{H+h}}{H+h}\right). 
\end{aligned}
\ee
Recall that $G^{H}_\alpha(\cdot)$ was defined in \eqref{g-set}. Using \eqref{drfit101} we get for all $h\in S_{H,\ol H}$,   
\be \label{kl1} 
\begin{aligned} 
&\ol P\big({G^{H}_\alpha(x)\cap G^{H}_\alpha(y)}\big) \\
&\leq P\bigg(X^h(x) \leq \frac{\alpha}{h+H} -  (1-2\dl)\gamma\Big( \frac{1}{H + h} +  \frac{1 - \norm{x-y}^{H+h}}{H+h}  \Big), \\
&\quad  \quad \quad X^h(y) \leq \frac{\alpha}{h+H} -  (1-2\dl)\gamma\Big( \frac{1}{H + h} +  \frac{1 - \norm{x-y}^{H+h}}{H+h}  \Big)  \bigg) \\
&=P\bigg(X^h(x) \leq (\alpha-\gamma(1-2\dl))\frac{1}{h+H}  - (1-2\dl)\gamma\frac{1 - \norm{x-y}^{H+h}}{H+h} , \\
&\quad  \quad \quad X^h(y) \leq  (\alpha-\gamma(1-2\dl))\frac{1}{h+H} - (1-2\dl)\gamma  \frac{1 - \norm{x-y}^{H+h}}{H+h}   \bigg).
\end{aligned} 
\ee
In order to bound the right hand side of \eqref{kl1}, we need to choose a specific $h^*$ from $S_{H,\ol H}$. 
Define 
\be \label{h-opt} 
h^{*}:=\left[ -\frac{\kappa^{*}}{\log \|x-y\|}\right]_{S}, 
\ee
where the constant $\kappa^{*}\geq 1$ will be specified later, and for any $x>0$, $[x]_{S}$ is the largest object in $S_{H,\ol H}$ that is smaller than $x$. 
Using \eqref{s-set} one can observe that for any arbitrary small $\dl >0$ we have for all $\ol H$ small enough and $(x,y)\in R^{2}$ that 
\be \label{gh1}
 -\frac{\kappa^{*}}{ \log\|x-y\|} -C\left( \frac{\kappa^{*}}{ \log\|x-y\|}\right)^{2} \leq h^{*}\leq  -\frac{\kappa^{*}}{\log \|x-y\|}, 
\ee
where $C>0$ is a constant independent from $x,y \in R_2$. 

It follows that for all $\ol H$ small enough 
\be \label{kl2} 
e^{-\kappa^{*}} \leq \|x-y\|^{h^{*}} \leq e^{-\frac{1}{2}\kappa^{*}}, \quad \textrm{for all } x,y \in R_2. 
\ee
Let $\alpha >\gamma$ where $\gamma <\gamma^*(d)$. Define 
\be   \label{eps-def} 
\eps = \alpha -\gamma. 
\ee
Then we have
\be \label{kl3} 
(\alpha-\gamma(1-2\dl))= \eps + 2\gamma\dl.  
\ee

Using \eqref{kl2} and \eqref{kl3} we get that
\be  \label{kl4} 
\begin{aligned} 
 (\alpha-\gamma(1-2\dl))\frac{1}{h^{*}+H} &=   \frac{(\alpha-\gamma(1-2\dl))}{1-e^{-\kappa^{*}/2}}\frac{1-e^{-\kappa^{*}/2}}{h^{*}+H} \\
 &\leq  \frac{(\alpha-\gamma(1-2\dl))}{1-e^{-\kappa^{*}/2}}\frac{1-\|x-y\|^{h^{*}}}{h^{*}+H} \\
  &\leq  \frac{(\alpha-\gamma(1-2\dl))}{1-e^{-\kappa^{*}/2}}\frac{1-\|x-y\|^{h^{*}+H}}{h^{*}+H} \\
    &\leq  \frac{\eps + 2\gamma\dl}{1-e^{-\kappa^{*}/2}}\frac{1-\|x-y\|^{h^{*}+H}}{h^{*}+H}.
\end{aligned} 
\ee
From \eqref{kl1} and \eqref{kl4} we get 
\be \label{p-med} 
\begin{aligned} 
\ol P\big({G^{H}_\alpha(x)\cap G^{H}_\alpha(y)}\big) 
& \leq P\bigg(X^h(x) \leq    -\beta (\dl,\eps) \gamma\frac{1 - \norm{x-y}^{H+h^{*}}}{H+h^{*}}, \\
&\qquad \qquad \qquad X^h(y) \leq   -\beta(\dl,\eps)\gamma  \frac{1 - \norm{x-y}^{H+h^{*}}}{H+h^{*}}   \bigg) \\
&=:P^{*}(x,y;H), 
\end{aligned} 
\ee
where 
\be \label{beta} 
\beta(\dl,\eps) := 1-2\dl\left(1 + \frac{1}{1-e^{-\kappa^{*}/2}}\right)- \frac{\eps}{\gamma(1-e^{-\kappa^{*}/2})}.
\ee

We would like to derive an upper bound on $P^{*}$.
 \begin{lemma} \label{lem-p-star}
Let $\eps = \alpha - \gamma$. Then, for any $\dl>0$ arbitrarily small there exists $\ol H$ small enough such that  
\bd
 P^{*}(x,y,H) \leq \frac{C}{\gamma^{2}}  \exp\Big(-(1-\dl)\beta(\dl,\eps)^{2}\gamma^{2}h^{*} \Big(  \frac{1 - \norm{x-y}^{H+h^{*}}}{H+h^{*}}\Big)^{2} \frac{2}{2-\|x-y\|^{2h^{*}}}\Big). 
\ed
Here $C>0$ is a constant not depending on $(\gamma, \eps, H, h^{*},\ol H, \dl)$.  
\end{lemma} 
The proof of Lemma \ref{lem-p-star} is postponed to Section \ref{sec-lem-p}. 

\begin{proof}[Proof of Proposition \ref{prop-j1}]  
Let $\dl >0$. From \eqref{j-i}, \eqref{p-med} and Lemma \eqref{lem-p-star} we get for $\ol H$ sufficiently small 
\be \label{rnd1}
\begin{aligned}
&J_{2}(H, \overline{H},S)\\
&\leq K \int\int_{R_{2}} \exp\Big(C_{H,H} \gamma^{2}\frac{1-\norm{x-y}^{2H}}{2H}\Big)P^{*}(x,y,H) dx dy \\
&\leq  \frac{C}{\gamma^{2}} \int \int_{R_{2}} \exp\Big((1+\dl) \gamma^{2}\frac{1-\norm{x-y}^{2H}}{2H}\Big) \\
& \qquad \quad  \times\exp\Big(-(1-\dl)\beta(\dl,\eps)^{2}\gamma^{2}h^{*} \Big(  \frac{1 - \norm{x-y}^{H+h^{*}}}{H+h^{*}}\Big)^{2} \frac{2}{2-\|x-y\|^{2h^{*}}}\Big) dx dy, 
\end{aligned}
\ee
where we have also used \eqref{c-h} in the last inequality. 

Assume now that for some $\ol \dl \in(0,1)$, $\xi \in (0,1/2)$ and $H_1 \in (0,H_0)$, we have for all $x,y \in R_2$ and $H \leq \ol H_1$
\be \label{must} 
 \gamma^{2}h^{*} \Big(  \frac{1 - \norm{x-y}^{H+h^{*}}}{H+h^{*}}\Big)^{2} \frac{2}{2-\|x-y\|^{2h^{*}}} 
> ( \xi +\ol \dl) \gamma^{2}\frac{1-\norm{x-y}^{2H}}{2H}. 
\ee
Then by choosing $\alpha$ close enough to $\gamma$, $\eps$ in \eqref{eps-def} now becomes arbitrarily small, and we get
$$
\frac{\eps}{\gamma(1-e^{-\kappa^{*}/2})} \leq \frac{\ol \dl}{16}. 
$$
By taking $\dl$ sufficiently small we have 
\be \label{eqr1}
2\dl\left(1 + \frac{1}{1-e^{-\kappa^{*}/2}}\right) \leq \frac{\ol \dl}{16}. 
\ee
It follows that $\beta(\dl,\eps)$ in \eqref{beta} is bounded from below by 
\be \label{rnd2}
1-\frac{\ol \dl}{8} \leq \beta(\dl,\eps). 
\ee
From \eqref{eqr1} we have $\dl \leq \ol \dl/64$. Together with \eqref{must} and \eqref{rnd2} we get 
\be \label{rnd3}
\begin{aligned}
& (1+\dl) \gamma^{2}\frac{1-\norm{x-y}^{2H}}{2H}  -(1-\dl)\beta(\dl,\eps)^{2}\gamma^{2}h^{*} \Big(  \frac{1 - \norm{x-y}^{H+h^{*}}}{H+h^{*}}\Big)^{2} \frac{2}{2-\|x-y\|^{2h^{*}}} \\
&\leq (1+\dl) \gamma^{2}\frac{1-\norm{x-y}^{2H}}{2H} - (1-\dl) \big(1-\frac{\ol \dl}{8} \big)^{2}( \xi +\ol \dl) \gamma^{2}\frac{1-\norm{x-y}^{2H}}{2H} \\
&\leq  \gamma^{2}\frac{1-\norm{x-y}^{2H}}{2H} \left(1+\dl-(1-\dl) \big(1-\frac{\ol \dl}{8} \big)^{2}( \xi +\ol \dl) \right) \\
&\leq  \gamma^{2}\frac{1-\norm{x-y}^{2H}}{2H}(1-\xi). 
\end{aligned}
\ee

Therefore from \eqref{rnd1} and \eqref{rnd3} for $\ol H$ small enough, which is depending on $\dl$ but not on $\eps$, we have
\be \label{j-2-b}
\begin{aligned}
\sup_{H\leq \ol H} J_{2}(H, \overline{H},S)& \leq C(\gamma) \int\int_{R_{2}} \exp\Big( (1-\xi) \gamma^{2}\frac{1-\norm{x-y}^{2H}}{2H}\Big) dx dy \\
&\leq C(\gamma)  \int\int_{ R_2} \exp\Big(-(1-\xi) \gamma^{2}\log\norm{x-y}\Big)  dx dy  \\
&\leq C(\gamma)  \int\int_{ \|x-y\| \leq 1} \norm{x-y}^{-(1-\xi) \gamma^{2}} dx dy \\
&<\infty, 
\end{aligned}
\ee
if $(1-\xi) \gamma^{2} <d$. Hence by assuming \eqref{must} we get
\be \label{hyop}
\sup_{H\leq \ol H} \sup_{S\in \mathcal S}J_{2}(H, \overline{H},S) <\infty, \quad \textrm{for all } \gamma^{2} < \frac{d}{1-\xi}, 
\ee
and the proof is complete. 

Therefore our goal is to show that \eqref{must} holds and to specify $\xi$. Define $\ol \xi =\xi+\ol \dl$. Since $\ol \dl \in(0,1)$ is arbitrarily small and $\xi \in (0,1/2)$,  \eqref{must} equivalent to  
\be \label{ineq1}
\frac{2\frac{h}{H}}{1-e^{2\log\norm{x-y}H}}\frac{1}{(1+\frac{h}{H})^{2}} (1 - e^{\log\norm{x-y}H(1+\frac{H}{h})})^{2} \frac{2}{2-e^{2H\log|x-y|\frac{h}{H}}} \geq \ol \xi, 
\ee
for some $\ol \xi \in (0,1/2)$. 

Substituting 
\be\label{change}
u = \frac{h}{H}, \qquad \lambda = -\log\norm{x-y}H, 
\ee
the left hand side of \eqref{ineq1} becomes
\be 
g(u,\lambda) := \frac{2u}{1-e^{-2\lambda}}\frac{1}{(1+u)^{2}}  \big(1 - e^{-\lambda(1+u)}  \big)^{2}\frac{2}{2-e^{-2\lam u}} . 
\ee
From the definition of $R_2$ in \eqref{region} and from  \eqref{gh1} it follows that for any arbitrarily small $\dl >0$ we have for all $\ol H$ small enough and $x,y\in R_{2}$ that
$$
 -\frac{\kappa^{*}}{H\log\|x-y\|} -\dl \leq \frac{h^{*}}{H}\leq  -\frac{\kappa^{*}}{H\log\|x-y\|} . 
$$ 
Together with \eqref{change} we get that 
\be \label{cont} 
 \frac{\kappa^{*}}{\lam}  -\dl \leq u \leq  \frac{\kappa^{*}}{\lam} . 
\ee
Now we fix $\kappa^{*}$. For any $\kappa>0$ let  
\begin{align*}
g( \frac{\kappa}{\lam}, \lambda) 
& = \frac{ 2\kappa}{\lam(1-e^{-2\lambda})}\frac{\lam^{2}}{(\lam +\kappa)^{2}}  \big(1 - e^{-(\kappa + \lam)} \big)^{2} \frac{2}{2-e^{-2\kappa}} \\
&= \frac{ 2\kappa \lam}{(1-e^{-2\lambda})}\frac{1}{(\lam +\kappa)^{2}}  \big(1 - e^{-(\kappa + \lam)} \big)^{2} \frac{2}{2-e^{-2\kappa}}.
\end{align*}
Define 
$$
f(\kappa):=\lim_{\lambda \to 0}g( \frac{\kappa}{\lam}, \lambda)  = \frac{2}{\kappa(2-e^{-2\kappa})} (1-e^{-\kappa})^{2}.
 $$ 
We further define 
$$
\kappa^{*}= \textrm{arg} \max_{\kappa} f(\kappa) \approx 1.0370, 
$$ 
and 
\be \label{xi} 
\bar \xi =f(\kappa^{*}) \approx 0.42872.
\ee

Note that 
$$
g( \frac{\kappa^*}{\lam}, \lambda)  \geq f(\kappa^*), \quad \textrm{ for all } 0<\lam \leq \kappa^{*}.
$$
So from \eqref{cont} and the continuity of $g$ we get that \eqref{ineq1} holds for $\ol \xi$ in \eqref{xi}. Note that \eqref{cont}  with $\ol \xi$ in \eqref{xi} holds any $x,y \in R_{2}$, since for such $x,y$ we have $\lam = -H\log\|x-y\|  \leq \kappa^{*}$ (see \eqref{region}). Therefore, for all $\ol H$ sufficiently small we get \eqref{hyop} where 
\be 
\gamma^*(d) =\sqrt{ \frac{d}{1-\ol \xi}} >\sqrt{1.75d}. 
\ee
\end{proof}

\section{Proof of Lemma \ref{lem-p-star}}  \label{sec-lem-p}

\begin{proof}[Proof of Lemma \ref{lem-p-star}]
Recall that $h^*$ was defined in \eqref{h-opt}. From \eqref{eq:XHCovariance} it follows that
$$
\big(X^{h^*}(x), X^{h^*}(y)\big) \sim N(0, \Sigma ),
$$
where
\be \label{sig-mat}
\Sigma  =\begin{bmatrix}
    C_{h^*,h^*}\big( \frac{1}{2h^{*}} +g^{h^{*},h^{*}}(x,x)   \big)  &  C_{h^{*},h^{*}}\big(\frac{1 - \norm{x-y}^{2h^{*}}}{2h^{*}}  +g^{h^{*},h^{*}}(x,y)\big)\\
   C_{h^{*},h^{*}}  \big( \frac{1 - \norm{x-y}^{2h^{*}}}{2h^{*}} +g^{h^{*},h^{*}}(y,x)\big)      &     C_{h^{*},h^{*}}\big( \frac{1}{2h^{*}}+g^{h^{*},h^{*}}(y,y) \big) \\
\end{bmatrix}.
\ee

By inverting $\Sigma$ we get 
\be\label{sig-inv}
\begin{aligned} 
&(\Sigma )^{-1} \\
&= \frac{1}{\det \Sigma} \begin{bmatrix}
     C_{h^{*},h^{*}}\big( \frac{1}{2h^{*}}+g^{h^{*},h^{*}}(y,y) \big)      & -C_{h^{*},h^{*}}\big(\frac{1 - \norm{x-y}^{2h^{*}}}{2h^{*}}  +g^{h^{*},h^{*}}(x,y)\big)  \\
   - C_{h^{*},h^{*}}  \big( \frac{1 - \norm{x-y}^{2h^{*}}}{2h^{*}} +g^{h^{*},h^{*}}(y,x)\big)     &    C_{h^{*},h^{*}}\big( \frac{1}{2h^{*}} +g^{h^{*},h^{*}}(x,x)   \big)  \\
\end{bmatrix}.
\end{aligned} 
\ee
From \eqref{c-h}, \eqref{eq:gEstimate} and \eqref{kl2} we get that for any arbitrarily small $\dl_{1}>0$ there exists $\ol H$ small enough (and hence $h^* \leq \ol H$ small) such that  
\be \label{det} 
\begin{aligned} 
\det \Sigma &\leq C^2_{h^{*},h^{*}}\frac{1}{4(h^{*})^2} \big(1-(1-\|x-y\|^{2h^{*}})^2\big)+ \frac{C}{h^{*}} \\ 
&\leq (1+\dl_{1})\frac{1}{4(h^{*})^2} \big(1-(1-\|x-y\|^{2h^{*}})^2\big),
\end{aligned} 
\ee
and similarly 
\be \label{det2} 
\begin{aligned} 
 \det \Sigma \geq (1-\dl_{1})\frac{1}{4(h^{*})^2} \big(1-(1-\|x-y\|^{2h^{*}})^2\big) .
\end{aligned} 
\ee

We will use the following bound which was derived by Savage in \cite{Savage}.
\begin{theorem}\label{thm-sav} 
Let $M= \Sigma^{-1} $ with $M=(m_{ij})_{d\times d}$ and $C=(c_1,...,c_d) \in \mathbb{R}^d$. If for all $1\leq i\leq d$, we have $\Delta_i := \sum_{j=1}^{d} C_j m_{ij}>0$ then 
$$
P(X_1 \geq c_1, ...., X_n \geq c_n) \leq \Big(\prod_{i=1}^d \Delta_i \Big)^{-1} \frac{\sqrt{\det{M}}}{(2\pi)^{d/2}} e^{-\frac{1}{2}C^TMC}.
$$ 
\end{theorem} 
Since $\{X^h(x)\}_{x\in \mathbb{R}^d}$ is a centred Gaussian field we get from \eqref{p-med} that,  
\bn
P^*:= P\Big( X^{h^{*}}(x) > \gamma \beta(\dl,\eps) \frac{1 - \norm{x-y}^{H+h^{*}}}{H+h^{*}} , \, X^{h^{*}}(y) > \gamma  \beta(\dl,\eps) \frac{1 - \norm{x-y}^{H+h^{*}}}{H+h^{*}} \Big).
\en
Using the notation of Theorem \ref{thm-sav} we have 
\be \label{c-def} 
C= (c,c) := \Big(\beta(\dl,\eps) \gamma  \frac{1 - \norm{x-y}^{H+h^{*}}}{H+h^{*}},\beta(\dl,\eps) \gamma   \frac{1 - \norm{x-y}^{H+h^{*}}}{H+h^{*}}\Big). 
\ee
Next we derive a lower bound to $\frac{1}{2}C^TMC$.  Using \eqref{eq:gEstimate}, \eqref{c-h-bnd}, \eqref{kl2} and \eqref{sig-inv}, we get for any $\dl >0$ arbitrarily small we have for all $\ol H$ small enough, 
\bn
\frac{1}{2}C^TMC&=&\frac{1}{2}c^{2}\sum_{i,j=1,2} M_{ij}  \\
&\geq &c^{2}(1-\dl)\frac{1}{\det \Sigma} \Big(\frac{1}{2h^{*}}-\frac{1 - \norm{x-y}^{2h^{*}}}{2h^{*}} -C_{1}\Big) \\
&\geq &c^{2}(1-2\dl)\frac{1}{\det \Sigma} \Big(\frac{1}{2h^{*}}-\frac{1 - \norm{x-y}^{2h^{*}}}{2h^{*}}  \Big) \\
&\geq &c^{2}\frac{(1-2\dl)}{(1+\dl_{1})}\frac{4(h^{*})^2}{\big(1-(1-\|x-y\|^{2h^{*}})^2\big)} \Big(\frac{1}{2h^{*}}-\frac{1 - \norm{x-y}^{2h^{*}}}{2h^{*}} \Big),
\en
where we have used \eqref{det} in the last inequity. 

Together with \eqref{c-def} we have 
\be \label{bf0} 
\begin{aligned} 
\frac{1}{2}C^TMC
&\geq \gamma^{2}\frac{\beta(\dl,\eps)^{2} (1-2\dl)}{(1+\dl_{1})} \frac{(1 - \norm{x-y}^{H+h^{*}})^{2}}{(H+h^{*})^{2}}\\ 
&\quad \times \frac{4(h^{*})^2}{\big(1-(1-\|x-y\|^{2h^{*}})^2\big)} \Big(\frac{1}{2h^{*}}-\frac{1 - \norm{x-y}^{2h^{*}}}{2h^{*}} \Big) \\
&\geq \gamma^{2} \beta(\dl,\eps)^{2}(1-\tilde \dl) \frac{2h^{*}}{2+\|x-y\|^{2h^{*}}}  \left( \frac{1 - \norm{x-y}^{H+h^{*}}}{H+h^{*}} \right)^{2}, 
\end{aligned} 
\ee
where $\tilde \dl > 0$ is arbitrarily small and depends on the choice of $\ol H$. 

Let $\dl_{2} \in (0,1)$. Recall that $M=\Sigma ^{-1}$. Using \eqref{sig-inv} and \eqref{c-def} and repeating similar steps as in the derivation of \eqref{bf0}, have for all $\ol H$ small enough 
$$
\Delta_1 =c(m_{11}+m_{12}) \geq (1- \dl_{2}) \beta(\dl,\eps)\gamma  \frac{1 - \norm{x-y}^{H+h^{*}}}{H+h^{*}}\frac{1}{\det \Sigma}\Big( \frac{1}{2h^{*}}  -\frac{1 - \norm{x-y}^{2h^{*}}}{2h^{*}} \Big),
$$
\be \label{ldel}
\Delta_2=c(m_{21}+m_{22}) \geq (1- \dl_{2})  \beta(\dl,\eps)\gamma  \frac{1 - \norm{x-y}^{H+h^{*}}}{H+h^{*}}\frac{1}{\det \Sigma}\Big( \frac{1}{2h^{*}}  -\frac{1 - \norm{x-y}^{2h^{*}}}{2h^{*}} \Big).
\ee
Note that by \eqref{kl2}, \eqref{det2} and \eqref{ldel}, $\Delta_1, \Delta_2 >0$ for $x\not =y$. 

Using \eqref{det} and (\ref{ldel}) and repeating the same lines as before, we get for arbitrarily small $\ol \dl>0$, 
\bn
\Delta_i &\geq &  2\gamma(1-\bar \dl)  \beta(\dl,\eps)h^{*}\|x-y\|^{2h^{*}}\frac{1 - \norm{x-y}^{H+h^{*}}}{H+h^{*}} \frac{1}{1-(1-\|x-y\|^{2h^{*}})^2 } \\
& \geq & 2\gamma(1-\bar \dl)  \beta(\dl,\eps)h^{*} \frac{1 - \norm{x-y}^{H+h^{*}}}{H+h^{*}} \frac{1}{2-\|x-y\|^{2h^{*}} },   \quad i=1,2. 
\en 
We therefore have 
\bn
(\Delta_1\Delta_2)^{-1} &\leq &   \Big(4(1-\bar \dl)^{2}\beta(\dl,\eps)^{2}\gamma^2  (h^{*})^2 \frac{(1 - \norm{x-y}^{H+h^{*}})^2}{(H+h^{*})^2} \frac{1}{(2-\|x-y\|^{2h^{*}})^2 }\Big)^{-1} \\ 
&\leq&   \frac{1}{4(1-\bar \dl)^{2}\beta(\dl,\eps)^{2}\gamma^2 (h^{*})^2}\frac{(H+h^{*})^2}{ (1 - \norm{x-y}^{H+h^{*}})^2} (2-\|x-y\|^{2h^{*}})^2. 
\en
Note that from \eqref{det2} we get 
\bn
\sqrt{\det M} = \frac{1}{\sqrt{\det \Sigma}} \leq C  \frac{2h^{*}}{\sqrt{1-(1-\|x-y\|^{2h^{*}})^2}}.
\en
We then have 
\be \label{bf1}
\Big(\prod_{i=1}^2 \Delta_i \Big)^{-1} \frac{\sqrt{\det{M}}}{2\pi}  \leq C \frac{1}{h^{*}\gamma^2}\frac{(H+h^{*})^2}{ (1 - \norm{x-y}^{H+h^{*}})^2} \frac{(2-\|x-y\|^{2h^{*}})^2}{\sqrt{1-(1-\|x-y\|^{2h^{*}})^2}}.
\ee

From \eqref{bf0} and \eqref{bf1} it follows that for $\ol H$ small enough we have 
\bn
P^{*} &\leq& C \frac{1}{h^{*}\gamma^2}\frac{(H+h^{*})^2}{ (1 - \norm{x-y}^{H+h^{*}})^2} \frac{(2-\|x-y\|^{2h^{*}})^2}{\sqrt{1-(1-\|x-y\|^{2h^{*}})^2}} \\
&&\times \exp\Big(-\beta(\dl,\eps)^{2}(1-\tilde \dl)h^{*} \gamma^{2}\Big(  \frac{1 - \norm{x-y}^{H+h^{*}}}{H+h^{*}}\Big)^{2} \frac{2}{2-\|x-y\|^{2h^{*}}}\Big).
\en
By using the following lemma we get the desired bound on $P^{*}$. 
\begin{lemma} \label{lem-bound} There exists $C>0$ such that 
$$
\sup_{H\leq h^*}\sup_{(x,y) \in R_{2}}  \frac{1}{h^{*}}\frac{(H+h^{*})^2}{ (1 - \norm{x-y}^{H+h^{*}})^2} \frac{(2-\|x-y\|^{2h^{*}})^2}{\sqrt{1-(1-\|x-y\|^{2h^{*}})^2}} <Ch^*.
$$
\end{lemma} 
\end{proof}

\begin{proof}[Proof of Lemma \ref{lem-bound}]
From \eqref{kl2} and since $H \leq h^{*}$ we get  
$$ 
 \norm{x-y}^{H+h^{*}}  \geq e^{-\kappa^{*}(H+h^{*})/h^{*}} \geq e^{-2\kappa^{*}}, 
$$
and similarly 
$$
\norm{x-y}^{2h^{*}} \geq  e^{-2\kappa^{*}}. 
$$
Therefore there exists $C>0$ independent from $h,H$ and $x,y$ such that 
\bn
\frac{1}{h^{*}}\frac{(H+h^{*})^2}{ (1 - \norm{x-y}^{H+h^{*}})^2} \frac{(2-\|x-y\|^{2h^{*}})^2}{\sqrt{1-(1-\|x-y\|^{2h^{*}})^2}} &\leq &C\frac{1}{h^{*}} (H+h^{*})^2 \\
&=& 4Ch^{*}.
\en
\end{proof}

\section{Convergence} \label{section-conv} 

In this section we prove the convergence of $\{M^{H}_\gamma\}_{H\in (0, H_0)}$ as $H \downarrow 0$. In order to do so, we will first show that for any $A\in \mathcal A $, $\{M^{H}_\gamma(A)\}_{H\in (0, H_0)}$ converges in $L^1$. We first describe our method of proof which uses ideas from \cite{berestycki2015elementary}.   

Recall that $M^{H}_\gamma$, $I^{H, \ol H}_{\alpha,\gamma}$ and $L^{H, \ol H}_{\alpha,\gamma}$ were defined in \eqref{approxMeasFBM}, (\ref{eq:def_I}) and \eqref{eq:def_L}, receptively.
Recall that by \eqref{m-dec} for any $\ol H \in (0,H_0)$ we have  
\be \label{m-1}
 M^{H}_\gamma(A)  =  I^{H, \ol H}_{\alpha,\gamma} (A) +  L^{H, \ol H}_{\alpha,\gamma}(A).
 \ee
Let $\eps>0$ be arbitrarily small, then by Corollary \ref{corr-supp} we can choose $\ol H$ small enough such that 
\be  \label{m-2}
 \sup_{0<H \leq \ol H/2}  \E \big[ L^{H, \ol H}_{\alpha,\gamma} (A)\big]  \leq  \frac{\eps}{2}.
 \ee
We will show that $\{I^{H, \ol H}_{\alpha,\gamma} (A)\}_{H \in (0,H_0)}$ is a Cauchy sequence in $L^2$, so we can choose $H_1 \in (0,\ol H)$ such that 
\be \label{m-3}
E\left[ \left(I^{H, \ol H}_{\alpha,\gamma} (A) - I^{H', \ol H}_{\alpha,\gamma} (A)\right)^2 \right] <\frac{\eps}{2}, \quad \textrm{for all } 0\leq H,H' \leq H_1.  
\ee
From \eqref{m-1}--\eqref{m-3} we get 
\be \label{m-4}
\begin{aligned} 
& E\left[ \left|M^{H, \ol H}_{\alpha,\gamma} (A) - M^{H', \ol H}_{\alpha,\gamma} (A)\right|\right] \\
&\leq E\left[ \left|I^{H, \ol H}_{\alpha,\gamma} (A) - I^{H', \ol H}_{\alpha,\gamma} (A)\right|\right] +E\left[ \left|L^{H, \ol H}_{\alpha,\gamma} (A) - L^{H', \ol H}_{\alpha,\gamma} (A)\right|\right]  \\  
&\leq \eps,  
\end{aligned} 
\ee
for all $0\leq H,H' \leq H_1$. Hence $\{M^{H}_\gamma(A)\}_{H\in (0, H_0)}$ is a Cauchy sequence in $L^1$, and this gives the convergence result. 
 
The remainder of this section is dedicated to showing that   $\{I^{H,\ol H}_{\alpha,\gamma}(A)\}_{H\in (0,\ol H)}$ converges in $L^{2}$ as $H\downarrow 0$. We summarise this result in the following proposition. 
\begin{proposition} \label{prop-l2} 
For any $\gamma <\gamma^*(d)$ and $\alpha >\gamma$ sufficiently close to $\gamma$, the set $\{I^{H, H_0}_{\alpha,\gamma}(A)\}_{H\in (0,\ol H)}$ converges in $L^{2}$ as $H \to 0$, for all $A\in \mathcal A$. 
\end{proposition}

We will show that $\{I^{H,\ol H}_{\alpha,\gamma}(A)\}_{H\in (0,\ol H)}$ is a Cauchy sequence in $L^{2}$, this will imply the convergence in Proposition \ref{prop-l2}. 

We first observe that for any $H, \hat H \in (0,\ol H)$ we have  
 \be \label{cs} 
 \begin{aligned} 
 E\Big[ \big( I^{H,\ol H}_{\alpha,\gamma}(A) - I^{\hat H,\ol H}_{\alpha,\gamma}(A)\big)^{2} \Big]& =  E\Big[  I^{H,\ol H}_{\alpha,\gamma}(A)^{2}\Big] -2 E\Big[  I^{H,\ol H}_{\alpha,\gamma}(A) I^{\hat H,\ol H}_{\alpha,\gamma}(A)\Big] \\
 &\quad+ E\Big[  I^{\hat H,\ol H}_{\alpha,\gamma}(A)^{2} \Big].
 \end{aligned} 
 \ee
 In the following two lemmas we derive a sharp upper on $E\Big[  I^{H,\ol H}_{\alpha,\gamma}(A)^{2}\Big]$ and a sharp lower bound on $E\Big[ I^{H,\ol H}_{\alpha,\gamma}(A)  I^{\hat H,\ol H}_{\alpha,\gamma}(A) \Big]$. These lemmas will help us to bound the right-hand side of \eqref{cs}.

 \begin{lemma} \label{lemma-con1}
 We have 
$$
\limsup_{H\rr 0} E\left[I^{H,\ol H}_{\alpha,\gamma}(A)^{2} \right] \leq \int_{A}\int_{A} e^{\gamma^{2}g(x,y)} \frac{1}{\|x-y\|^{\gamma^{2}}} g_{\alpha}(x,y)dx dy, 
$$
where $ g_{\alpha}$ is a nonnegative function depending on $\alpha, \ol H$ and $\gamma$.   
\end{lemma} 
\begin{proof} 
We fix $\eta \in (0,e^{-2/\ol H})$. Note that from \eqref{IHL2Bound} we have 
\be \label{f1}
\begin{aligned} 
 &E\left[I^{H,\ol H}_{\alpha,\gamma}(A)^{2} \right]  \\
 &\leq K \int_A\int_A \mathbf{1}_{\|x-y\| \leq \eta} \exp\Big(C_{H,H} \gamma^{2}\frac{1-\norm{x-y}^{2H}}{2H}\Big)\,  \ol P(G^{H}_\alpha(x) \cap G^{H}_\alpha(y)) dx dy \\
 &\quad+ K \int_A\int_A  \mathbf{1}_{\|x-y\| \geq \eta} \exp\Big(C_{H,H} \gamma^{2}\frac{1-\norm{x-y}^{2H}}{2H}\Big)\,  \ol P(G^{H}_\alpha(x) \cap G^{H}_\alpha(y)) dx dy \\
 &=: K\big(I_{1}(\eta,H) + I_{2}(\eta,H)\big). 
 \end{aligned} 
\ee
Since $\eta <e^{-2/\ol H} $ from \eqref{region} it follows that 
\be \label{i1bl} 
\begin{aligned} 
I_{1}(\eta,H) &\leq \iint_{R_{1}\cup R_{2}}  \mathbf{1}_{\|x-y\| \leq \eta} \exp\Big(C_{H,H} \gamma^{2}\frac{1-\norm{x-y}^{2H}}{2H}\Big)\,  \ol P(G^{H}_\alpha(x) \cap G^{H}_\alpha(y)) dx dy\\
&= \iint_{R_{1}}  \mathbf{1}_{\|x-y\| \leq \eta} \exp\Big(C_{H,H} \gamma^{2}\frac{1-\norm{x-y}^{2H}}{2H}\Big)\,  \ol P(G^{H}_\alpha(x) \cap G^{H}_\alpha(y)) dx dy\\
&\quad +\iint_{  R_{2}}  \mathbf{1}_{\|x-y\| \leq \eta} \exp\Big(C_{H,H} \gamma^{2}\frac{1-\norm{x-y}^{2H}}{2H}\Big)\,  \ol P(G^{H}_\alpha(x) \cap G^{H}_\alpha(y)) dx dy \\
& =: I_{1,1}(\eta,H) +I_{1,2}(\eta,H). 
\end{aligned} 
\ee

Using \eqref{j-i} and \eqref{j-1-b} we get 
\be  \label{gjk1}
I_{1,1}(\eta,H) \leq  C\int\int_{  \{\|x-y\|\leq \eta \wedge e^{-\kappa^{*}/H} \}} \exp\Big((1+\dl) \gamma^{2}\frac{1}{2H}\Big)dx dy. 
\ee
From \eqref{j-i} and \eqref{j-2-b} we have 
\be \label{gjk2}
 I_{1,2}(\eta,H) \leq C   \int\int_{ \|x-y\| \leq \eta} \norm{x-y}^{-(1-\xi) \gamma^{2}} dx dy, 
\ee
where $\xi$ is given by \eqref{xi} and is chosen so that the right-hand side of \eqref{gjk2} is finite for any $\eta \leq 1$ and for $\gamma \leq \gamma^*(d)$. 

By plugging in \eqref{gjk1} and \eqref{gjk2} to \eqref{i1bl}, it follows that there exists a function $\ell(\eta)$ such that for all $\gamma <\gamma^{*}(d)$ we have 
\be \label{I-1-bf}
 \sup_{H\leq \ol H}I_{1}(\eta,H) \leq \ell(\eta), \quad \textrm{where } \ell(\eta) \rr 0 \textrm{ as } \eta \rr 0. 
\ee

Next we bound $I_{2}(\eta,H)$. We will need the following lemma, that follows immediately from \eqref{eq:XHCovariance} and \eqref{c-h}.
\begin{lemma}\label{mean-lem} For any fixed $H_{1} \in (0, \ol H/2)$ we have 
\begin{itemize} 
\item[1.] $$
\lim_{H\rr 0 } \sup_{ x\in A, \, h\geq H_{1}} \left| E\big[X^{h}(x) X^{H}(x)\big]  - C_{h,0}\left( \frac{1}{h} + g^{h}(x,x) \right) \right|, 
$$
\item[2.] 
 $$
\lim_{H\rr 0 } \sup_{\|x-y\|\geq \eta, \, h\geq H_{1}} \left| E\big[X^{h}(x) X^{H}(y)\big]  - C_{h,0}\left( \frac{1-\|x-y\|^{h}}{h} + g^{h}(x,y) \right) \right|.
$$
\end{itemize} 
\end{lemma} 
Recall that $\ol P$ (which depends on $H$) was defined in \eqref{ol-p}. By Lemma \ref{mean-lem} and the Cameron-Martin-Girsanov theorem the joint law of $(X^{h}(x), X^{h}(y))_{h \in (0,\ol H] }$ converges as $H\downarrow 0$ under $\ol P$ to a joint distribution  $(\wt X^{h}(x), \wt X^{h}(y))_{h \in(0,\ol H]}$ with the same covariance structure, but with drifts which are given by 
\be \label{drift1}
\begin{aligned} 
E\big[X^{h}(x)\big]   &= \gamma C_{h,0}\left( \frac{2-\|x-y\|^{h}}{h} + g^{h}(x,x) +g^{h}(x,y)\right), \\
E\big[X^{h}(y)\big]   &= \gamma C_{h,0}\left( \frac{2-\|x-y\|^{h}}{h} + g^{h}(y,y) +g^{h}(x,y) \right). 
\end{aligned}
\ee 
This weak convergence holds uniformly on compacts of $(0,\ol H]$ and on $\|x-y\|\geq \eta $. 

Let 
\be \label{gt-set} 
\wt G_{\ol H}(x)  = \left\{\wt X^{h}(x) \leq  \frac{\alpha}{h}+ \gamma C_{h,0}\left( \frac{2-\|x-y\|^{h}}{h} + g^{h}(x,x) +g^{h}(x,y)\right), \, \forall \, 0<h \leq  \ol H \right\}.
\ee
Then from \eqref{g-set}, \eqref{drift1} and \eqref{gt-set}, we get uniformly on $\|x-y\|>\eta$ we have 
\be \label{d1} 
\lim_{H \rr 0}\ol P\left(G_{H,\ol H}(x) \cap  G_{H,\ol H}(y)\right) =  P(\wt G_{\ol H}(x) \cap \wt G_{\ol H}(y))  := g_{\alpha}(x,y). 
\ee

Using \eqref{eq:XHCovariance}--\eqref{c-h} we get uniformly in $\|x-y\|\geq \eta$ 
\be \label{d222} 
\lim_{H\rr 0} E[X^{H}(x) X^{H}(y)] = -\log\|x-y\| + g^{}(x,y).
\ee
Then from \eqref{d1} and \eqref{d222} and since $g$ is bounded, we can use dominated convergence to get 
\be \label{g1}
\begin{aligned} 
&\lim_{H\rr 0} \int_{A}\int_{A} \mathbf{1}_{\|x-y\| \geq \eta} e^{\gamma^{2}E[X^{H}(x) X^{H}(y)]}\ol P\left(G_{H,\ol H}(x) \cap  G_{H,\ol H}(y)\right)dx dy \\
& = \int_{A}\int_{A} \mathbf{1}_{\|x-y\| \geq \eta} e^{\gamma^{2}g(x,y)} \frac{1}{\|x-y\|^{\gamma^{2}}} g_{\alpha}(x,y)dx dy. 
\end{aligned} 
\ee
To finish the proof we need to show that the right hand side of \eqref{g1} is finite when $\eta \rr 0$. 

Note that from \eqref{p-med} any Lemma \ref{lem-p-star}, for and $\dl >0$ there exists $\ol H$ sufficiently small, such that for all $H\leq \ol H$ we have 
$$
\begin{aligned} 
&\ol P\left(G_{H,\ol H}(x) \cap  G_{H,\ol H}(y)\right) \\
&\leq \frac{C}{\gamma^{2}}  \exp\Big(-(1-\dl)\beta(\dl,\eps)^{2}\gamma^{2}h^{*} \Big(  \frac{1 - \norm{x-y}^{H+h^{*}}}{H+h^{*}}\Big)^{2} \frac{2}{2-\|x-y\|^{2h^{*}}}\Big). 
\end{aligned} 
$$
where $h^{*}$ was defined in \eqref{h-opt} and $\eps=\alpha-\gamma$. From \eqref{must}, \eqref{xi} and by choosing $\dl$ small enough and $\alpha$ close to $\gamma$, we have we have for all $H\leq \ol H$, 
\bd   
(1-\dl)\beta(\dl,\eps)^{2} \gamma^{2}h^{*} \Big(  \frac{1 - \norm{x-y}^{H+h^{*}}}{H+h^{*}}\Big)^{2} \frac{2}{2-\|x-y\|^{2h^{*}}} 
> \gamma^{2} f(\kappa^{*}) \frac{1-\norm{x-y}^{2H}}{2H}. 
\ed

It follows that 
\be \label{fgh} 
\begin{aligned} 
\lim_{H\rr 0}\ol P\left(G_{H,\ol H}(x) \cap  G_{H,\ol H}(y)\right) 
&\leq \lim_{H\rr 0} \frac{C}{\gamma^{2}}  \exp\Big(-   f(\kappa^{*}) \frac{1-\norm{x-y}^{2H}}{2H} \Big) \\ 
&=   \frac{C}{\gamma^{2}}  \exp\Big(\gamma^{2}  f(\kappa^{*}) \log\|x-y\|\Big).
\end{aligned} 
\ee
Therefore using \eqref{d1}, \eqref{fgh}, \eqref{j-2-b} and \eqref{xi}, we get for all $\gamma <\gamma^{*}(d)$, 
$$
\begin{aligned} 
&\sup_{\eta \in  (0,e^{-2/\ol H})}\int_{A}\int_{A} \mathbf{1}_{\|x-y\| \geq \eta} e^{\gamma^{2}g(x,y)} \frac{1}{\|x-y\|^{\gamma^{2}}} g_{\alpha}(x,y)dx dy \\ 
&\leq \sup_{\eta \in  (0,e^{-2/\ol H})} C(\gamma)\int_{A}\int_{A} \mathbf{1}_{\|x-y\| \geq \eta}  \frac{1}{\|x-y\|^{\gamma^{2}}}  \exp\Big(\gamma^{2}  f(\kappa^{*}) \log\|x-y\|\Big)dx dy \\ 
&\leq C(\gamma)\int_{A}\int_{A}  \frac{1}{\|x-y\|^{\gamma^{2}(1-f(\kappa^{*}))}}   dx dy \\ 
&<\infty. 
\end{aligned} 
$$
This proves that the right hand side of \eqref{g1} is finite when $\eta \rr 0$, and therefore the conditions of dominated convergence apply. 
It follows from \eqref{g1} and \eqref{f1} that  
$$
\lim_{\eta \rr 0} \lim_{H\rr0} I_2(\eta,H) \leq \int_{A}\int_{A} e^{\gamma^{2}g(x,y)} \frac{1}{\|x-y\|^{\gamma^{2}}} g_{\alpha}(x,y)dx dy. 
$$

Together with \eqref{f1} and \eqref{I-1-bf} this completes the proof. 
\end{proof} 
 
\begin{lemma} \label{lem-lb}
 We have 
$$
\liminf_{H, \hat H \rr 0} E\left[I^{H,\ol H}_{\alpha,\gamma}(A)I^{ \hat H,\ol H}_{\alpha,\gamma}(A)\right] \geq \int_{A}\int_{A} e^{\gamma^{2}g(x,y)} \frac{1}{\|x-y\|^{\gamma^{2}}} g_{\alpha}(x,y)dx dy. 
$$
\end{lemma} 
\begin{proof} The proof is almost identical to the proof of Lemma \ref{lemma-con1}. Repeating the same steps leading to \eqref{IHL2Bound} we get tor any $H, \hat H \in (0,\ol H)$  
$$
\begin{aligned} 
&E\left[I^{H,\ol H}_{\alpha,\gamma}(S)I^{ \hat H,\ol H}_{\alpha,\gamma}(A)\right]   \geq \int_{A}\int_{A} \mathbf{1}_{\|x-y\| \geq \eta} e^{\gamma^{2}E[X^{H}(x) X^{\hat H}(y)]} \hat P\left(G_{H,\ol H}(x) \cap  G_{\hat H,\ol H}(y)\right)dx dy \\
\end{aligned} 
$$
where 
\be \label{hat-p} 
\frac{d \hat {P}}{dP}=e^{\gamma X^{H}(x) + \gamma X^{\hat  H}(y) - \frac{\gamma^{2}}{2}E[X^{H}(x) + X^{\hat H}(y)]^{2}}.
\ee
Again, the joint law of $(X^{h}(x), X^{h}(y))_{h \leq \ol H/2 }$ converges when $H$ and $\hat H$ tend to $0$ under $\hat P$ to  a joint distribution $(\wt X^{h}(x), \wt X^{h}(y))_{h \leq \ol H/2}$ that has the same covariance structure but with drift which is given by \eqref{drift1}.  
This weak convergence is uniform on compacts of $(0,\ol H]^{2}$ and on $\|x-y\|\geq \eta $. 

Recall that $\wt G_{\ol H}(x) $ was defined in \eqref{gt-set}. Then uniformly in $\|x-y\|>\eta$ 
\be\label{bb1}
\lim_{\bar H, H \rr 0}\hat P\left(G_{H,\ol H}(x) \cap  G_{\hat H,\ol H}(y)\right) =  P(\wt G_{\ol H}(x) \cap \wt G_{\ol H}(y))  := g_{\alpha}(x,y). 
\ee
Using \eqref{eq:XHCovariance}--\eqref{c-h} we get uniformly in $\|x-y\|\geq \eta$ 
\be \label{d2} 
\lim_{H, \hat H \rr 0} E[X^{H}(x) X^{\hat H}(y)] = -\log\|x-y\| + g^{}(x,y).
\ee
Since $g$ is bounded on $D\times D$ and from \eqref{bb1}, \eqref{d2}, we can use dominated convergence to get 
$$
\begin{aligned} 
&\liminf_{H, \hat H\rr 0} \int_{A}\int_{A} \mathbf{1}_{\|x-y\| \geq \eta} e^{\gamma^{2}E[X^{H}(x) X^{\hat H}(y)]} \hat P\left(G_{H,\ol H}(x) \cap  G_{\hat H,\ol H}(y)\right)dx dy \\
& \geq \int_{A}\int_{A} \mathbf{1}_{\|x-y\| \geq \eta} e^{\gamma^{2}g(x,y)} \frac{1}{\|x-y\|^{\gamma^{2}}} g_{\alpha}(x,y)dx dy. 
\end{aligned} 
$$
Since $\eta$ is arbitrarily small, the result follows. 
\end{proof} 

Now we are ready to prove Proposition \ref{prop-l2}. 
\begin{proof} [Proof of Proposition \ref{prop-l2}] 
Let $A\in \mathcal A$. From \eqref{cs} and Lemmas \ref{lemma-con1} and \ref{lem-lb} it follows that $\{I^{H, H_0}_{\alpha,\gamma}(A)\}_{H\in (0,\ol H)}$ is a Cauchy sequence in $L^{2}$. 
\end{proof}

Next we show that Proposition \ref{prop-l2} implies the convergence of $\{M^{H}_{\alpha,\gamma}(A)\}_{H\in (0,\ol H)}$ as $H\downarrow 0$ in $L^{1}$. 

\begin{proof} [Proof of convergence in Theorem \ref{thm-convergence}]
Let $A\in \mathcal A$. From Proposition \ref{prop-l2} it follows that $\{I^{H, H_0}_{\alpha,\gamma}(A)\}_{H\in (0,\ol H)}$ is a Cauchy sequence in $L^{2}$. From the explanation at the beginning of this section (see \eqref{m-1}--\eqref{m-4}) it follows that $\{M^{H}_{\gamma}(A)\}_{H\in (0,H_0)}$ converges in $L^{1}$, and therefore it converges  in probability to a limit $M_{\gamma}(A)$, when $\gamma<\gamma^{*}(d)$. The next step is to show that the sequence of measures $\{M^{H}_{\gamma}\}_{H\in (0,\ol H)}$ converges in probability with the weak topology towards a measure $M_{\gamma}$, for $\gamma<\gamma^{*}(d)$. This procedure is identical to the corresponding argument in Section 6 of \cite{berestycki2015elementary}, hence it is omitted. 
\end{proof}

\section{Proof of Lemma \ref{lem:proof_fbf_covariance}} \label{corr-comp}

The values of the constants $h_{h,H}$, $o_{h,H}$, $m_H$, $k_H^d$, $C_{H,h}^d$ that appear in the section are given in Appendix \ref{appendix}.  

\begin{proof} [Proof of Lemma \ref{lem:proof_fbf_covariance}]
It was already shown in \cite[Lemma 1]{lindstrom1993fractional} that the integral in \eqref{eq:fbm_construction_rd} is well defined. 

We therefore only need to prove the explicit form of the covariance structure. 
The proof uses ideas from the proof of \cite[Lemma 3]{lindstrom1993fractional}.
Let $x,y \in\Rd$ and $H,h\in(0,1)$. Then Itô-isometry we have 
\begin{align*}
E\big[\big(B^H(x)-&B^H(y)\big)\big(B^h(x)-B^h(y)\big)\big] \\ 
=&k_H^{d}k_h^{d} \int_{\Rd}\big(\norm{x-u}^{H-\frac{d}{2}} - \norm{y-u}^{H-\frac{d}{2}}\big)\big(\norm{x-u}^{h-\frac{d}{2}} - \norm{y-u}^{h-\frac{d}{2}}\big)du.
\end{align*}
Now consider the following substitution
\begin{align*}
w = \frac{u-y}{\norm{x-y}} \aand \norm{x-u}=\norm{x-y}\norm{e - w},
\end{align*}
where $e:=(x-y)/\norm{x-y}$ is a unit vector.
Note that 
$$
\norm{y-u}=\norm{x-y}\norm{w}.
$$
We therefore get that
\begin{align*}
E\big[\big(B^H(x)-&B^H(y)\big)\big(B^h(x)-B^h(y)\big)\big] = \tilde{c}_{H,h}^d \norm{x-y}^{H+h},
\end{align*}
where,
\be \label{eqr0}
\begin{aligned}
\tilde{c}_{H,h}^d := k_H^{d}k_h^{d}\int_{\Rd}\big(\norm{e -w}^{H-\frac{d}{2}} - \norm{w}^{H-\frac{d}{2}}\big)\big(\norm{e-w}^{h-\frac{d}{2}} - \norm{w}^{h-\frac{d}{2}}\big)dw.
\end{aligned}
\ee
Due to the rotational invariance of the integral, $\tilde{c}_{H,h}^d$ is indeed independent of the orientation of the unit vector $e$.
In the following we fix $e=e_1=(1,0,...,0)$. Therefore, in order to complete the proof we need to show that $\tilde{c}_{H,h}^d= c_{H,h}^d$.

We first recall the definition of the Riesz-kernel $\mathfrak{r}^{\alpha}_d: \Rd \to \R$ (see \cite[Chapter 1.1]{landkof1972foundations}). 
\begin{align*}
\mathfrak{r}^{\alpha}_d (x) = A_d(\alpha)\norm{x}^{\alpha - d} \with A_d(\alpha) = \frac{2^{\frac{d}{2} - \alpha}\Gamma(\frac{d-\alpha}{2})}{\Gamma(\frac{\alpha}{2})},
\end{align*}
and the Riesz-potential $I^{\alpha}_d$
\begin{align*}
    (I^{\alpha}_d\phi)(x) = \int_{\Rd}\mathfrak{r}_d^{\alpha}(x-u)\phi(u)du, \quad \phi\in\mathcal{S},
\end{align*}
for any $0 < \alpha < d$ and for any $\alpha\in\mathbb{C}$ with $\alpha\neq d + 2k, -2k$, $k\in \mathbb{N}$, by analytic continuation (see \cite[Chapter 1.1.2]{landkof1972foundations}).
The Fourier-transform of Riesz-Potential in the sense of distributions is given by
\begin{align*}
\int_{\Rd} \mathfrak{r}^\alpha_d(x)\hat{\phi}(x)dx = \int_{\Rd}\norm{\xi}^{-\alpha}\phi(\xi)d\xi, \quad \textrm{for all } \phi \in \mathcal S,
\end{align*}
where $$\hat{\phi}(\xi) = \mathcal{F}[\phi](\xi) := (2\pi)^{-\frac{d}{2}}\int_{\Rd} e^{-ix\xi} \phi(x)dx.$$
We then have by the linearity of the Fourier-transform that
\begin{align*}
    \int_{\Rd}(\mathfrak{r}^{\alpha}_d(x-e_1)- \mathfrak{r}^{\alpha}_d(x))\hat{\phi}(x)dx = \int_{\Rd}\frac{1-e^{-i\xi e_1}}{\norm{\xi}^{\alpha}}\phi(\xi)d\xi, \quad \phi \in \mathcal{S}.
\end{align*}
However, as we have discussed earlier, the function $x \mapsto (\mathfrak{r}^{\alpha}_d(x-e_1)- \mathfrak{r}^{\alpha}_d(x))$ with $\alpha = H + \frac{d}{2}$ is in $L^{2}(\Rd)$ and therefore its Fourier-transform in the sense of distributions coincides with its Fourier-transform in the sense of $L^{2}$-functions.
We therefore have
\begin{align*}
(2\pi)^{-\frac{d}{2}}\int_{\Rd}\big(\norm{e_1 - w}^{H-\frac{d}{2}} - \norm{w}^{H-\frac{d}{2}}\big)e^{-i\xi w}du = \frac{1}{A_d(H+\frac{d}{2})}\frac{1-e^{-i\xi e_1}}{\norm{\xi}^{H+\frac{d}{2}}},
\end{align*}
for all $\xi \in\Rd, \xi \neq 0$ and it follows from the Plancherel theorem 
that
\begin{align*}
\tilde{c}_{h,H}^{d}=\frac{k_H^{d}k_h^{d}}{A_d(H+\frac{d}{2})A_d(h+\frac{d}{2})} \int_{\Rd} \frac{\abs{1-e^{-i\xi e_1}}^{2}}{\norm{\xi}^{H+h +d}}d\xi.
\end{align*}
Finally, from the evaluation above integral which is given  in Lemma \ref{lem:spherical_riesz_integral} below, we obtain
\begin{align*}
\tilde{c}^{d}_{H,h} 
=& 2\pi^{\frac{d+1}{2}}\frac{k_H^{d}k_h^{d}}{A_d(H+\frac{d}{2})A_d(h+\frac{d}{2})} \frac{\Gamma(\frac{H+h+1}{2})}{\Gamma(\frac{H+h+d}{2})}  \frac{1}{(H+h)\Gamma(H+h)\sin(\frac{H+h}{2}\pi)} \\
=& 2\pi^{\frac{d+1}{2}}\frac{k_H^{d}k_h^{d}}{\Gamma(\frac{d}{4}-\frac{H}{2})\Gamma(\frac{d}{4}-\frac{h}{2})} \frac{\Gamma(\frac{H+h+1}{2})}{\Gamma(\frac{H+h+d}{2})}  \frac{2^{H+h}\Gamma(\frac{d}{4}+\frac{H}{2})\Gamma(\frac{d}{4}+\frac{h}{2})}{(H+h)\Gamma(H+h)\sin(\frac{H+h}{2}\pi)} \\
=& c_{h,H}^{d}.
\end{align*}
\end{proof}

\begin{lemma}\label{lem:spherical_riesz_integral}  
For any $h\in (0,\frac{1}{2})$ we have 
\begin{align*}
\int_{\Rd} \frac{\abs{1-e^{-i\xi e_1}}^{2}}{\norm{\xi}^{2h + d}} d\xi = \pi^{\frac{d+1}{2}}  \frac{\Gamma(h + \frac{1}{2})}{\Gamma(h + \frac{d}{2})h\Gamma(2h)\sin(h\pi)}.
\end{align*} 
\end{lemma}
\begin{proof}
The case where $d=1$ was proved in \cite[Chapter 7.2, Proposition 7.28]{samoradnitsky2017stable} and the case where $d\ge 2$ was proved in \cite[Section 3.6]{lacaux2004multifractional}.
\end{proof}

\section{Proof of Theorem \ref{thm:fbm_normalization}}\label{sec:proof_normalization}
In order to prove Theorem \ref{thm:fbm_normalization}, we will show that the integral (\ref{def:XH}), that is 
$$
I^H(x)=\int_{\Rd} B^{H}(u) \psi(u, x)du, \quad x \in \re^d, 
$$
is well defined, almost surely finite and Gaussian and that the normalized filed has the covariance structure as in (\ref{eq:XHCovariance}). Then will prove that the Gaussian fields in \eqref{eq:mandelbrot_van_ness} and \eqref{eq:fbm_construction_rd} after normalization, satisfy Assumption \ref{asmp}. 

Note that $\{\tilde B^H\}_{H \in (0,1)}$ in \eqref{eq:mandelbrot_van_ness} and  $\{ B^H\}_{H \in (0,1/2)}$ in \eqref{eq:one_dim_symmetric} and \eqref{eq:fbm_construction_rd} satisfy the covariance relation \eqref{cov-mv} and \eqref{cov-fbf}, respectively. In what follows we will take the interval of $H$ in the construction above to be $0\leq H< H_0$, where $H_0=1$ for \eqref{eq:mandelbrot_van_ness} and $H_0=1/2$ for \eqref{eq:one_dim_symmetric} and \eqref{eq:fbm_construction_rd}.

\textbf{Step 1: Properties of $I^H$.} It is well known that fractional Brownian fields are almost surely H\"older continuous and in particular measurable on $\Rd$ \cite[Chapter 8.3, Theorem 8.3.2]{adler2010geometry}.
Moreover, for any $\eps >0$ there exists an almost surely finite random variable $K_\eps$ such that
\be \label{poi1}
\abs{B^{H}(x)} \le K_\eps(1 + \norm{x}^{H+\eps}) \foralll x\in\Rd, \ \text{a.s.},
\ee
(see \cite[Lemma 5 and Remark 5]{kozachenko2011drift}).  

Recall that the class of normalizing functions $\mathcal N_{H_0}(D)$ was defined in Definition \ref{def-ker}.
Let $\psi \in \mathcal N_{H_0}(D)$. From (\ref{psiConditionOne}), (\ref{psiConditionTwo}) and \eqref{poi1} it follows that
\begin{align*}
\int_{\Rd}\abs{B^{H}(u)}\psi(u,y)du  <\infty, \quad \textrm{for all } x\in D, \ 0<H<H_0. \quad\text{a.s.}
\end{align*}

Hence the integral in (\ref{def:XH}) is well defined and almost surely finite.
In order to show that it is also Gaussian, we use a standard a  Riemann-sums approximation, which are clearly Gaussian.

\textbf{Step 2: Covariance structure. } We first prove that the normalized field $X^H$ in \eqref{def:XH} has the covariance structure (\ref{eq:XHCovariance}) for $\{B^H\}_{H\in (0,H_0)}$ in (\ref{eq:one_dim_symmetric}) and (\ref{eq:fbm_construction_rd}).

Recall that in both these cases we have 
\be \label{cov-unif}
\E\big[ B^{H}(x) B^{h}(y) \big] = c^d_{H, h}\big( \norm{x}^{H + h} + \norm{y}^{H+h} - \norm{x-y}^{H+h}\big).
\ee
In order to calculate the covariance of $X^H$ we use of Fubini's theorem, so that expectation and the integral in \eqref{def:XH} can be interchanged. We first verify the essential integrability condition.    
By Young's inequality we have 
\begin{align*}
\E\big[\abs{B^{H}(x) B^{h}(y)} \big] \le \frac{1}{2}\E\big[B^{H}(x)^{2} + B^{h}(y)^{2} \big] = C(H, h)\left(\norm{x}^{2H} + \frac{1}{2}\norm{y}^{2h}\right). 
\end{align*}
Then from (\ref{psiConditionOne}) and (\ref{psiConditionTwo}) and \eqref{cov-unif} we get for all $0 < H,h < H_0$ and $x,y \in D$,
\begin{align*}
\int_{\Rd}\E\big[\abs{B^{H}(x) B^{h}(u)} \big] \psi(u,y)du &   < \infty, \\
\int_{\Rd}\int_{\Rd}\E\big[\abs{B^{H}(v) B^{h}(u)} \big] \psi(u,y) \psi(v,x)du & < \infty.
\end{align*}

Using \eqref{def:XH} and Fubini's theorem we get 
\be  \label{cov-comp}
\begin{aligned}
&\frac{1}{(\Gamma(H)\Gamma(h))^{\frac{1}{2}}}\E\big[ X^{H}(x)X^{h}(y)\big] \\
&=\; \E\big[B^{H}(x) B^{h}(y)\big] - \int_{\Rd}\E\big[B^{H}(x) B^{h}(u)\big] \psi(u,y) du  \\
&\quad - \int_{\Rd}\E\big[B^{H}(y) B^{h}(u)\big] \psi(u,x) du  
+ \int_{\Rd}\int_{\Rd} \E\big[B^{H}(u) B^{h}(v)\big] \psi(u,y)\psi(v,x) du dv \\
&:=\sum_{i=1}^4L_i(x,y). 
\end{aligned}
\ee
Using \eqref{cov-unif} and then (\ref{psiConditionOne}) we get
\begin{align*}
\frac{ L_2(x,y)}{c_{H,h}^d}
=& -\int_{\Rd}\big( \norm{x}^{H + h} + \norm{u}^{H+h} - \norm{x-u}^{H+h} \big) \psi(u,y) du \\
=& -\norm{x}^{H + h} - \int_{\Rd} \norm{u}^{H+h} \psi(u,y) du + \int_{\Rd} \norm{x-u}^{H+h}  \psi(u,y) du 
\end{align*}
and by symmetry we have 
\begin{align*}
\frac{ L_3(x,y)}{c_{H,h}^d}=\frac{ L_2(y,x)}{c_{H,h}^d}.
\end{align*}
By the same argument we have 
\begin{align*}
\frac{L_4(x,y)}{c^d_{H,h}} =& \underset{\Rd}{\int} \norm{u}^{H+h} \psi(u,y) du + \underset{\Rd}{\int} \norm{v}^{H+h} \psi(v,x) dv \\
&-\underset{\Rd\times\Rd}{\iint} \norm{v-u}^{H+h} \psi(u,y)\psi(v,x) dudv.
\end{align*}
Plugging in the expressions for $L_i$, $i=1,...,4$, to \eqref{cov-comp}, we get for all $x,y \in D$ that 
\be \label{gg1} 
\begin{aligned}
\frac{\E\big[ X^{H}(x)X^{h}(y)\big]}{c_{H,h}^d\sqrt{\Gamma(H)\Gamma(h)}} =& \norm{x-y}^{H+h} + \int_{\Rd} \norm{x-u}^{H+h}  \psi(u,y) du \\
&+ \int_{\Rd} \norm{y-v}^{H+h}  \psi(v,x) dv\\
&-\underset{\Rd\times\Rd}{\iint} \norm{v-u}^{H+h} \psi(u,y)\psi(v,x) dudv.
\end{aligned}
\ee
Define
\be \label{gg2} 
\begin{aligned}
C_{H,h} = c_{H,h}^d\sqrt{\Gamma(H)\Gamma(h)}(H+h), 
\end{aligned}
\ee
and 
\be  \label{gg3} 
\begin{aligned}
g^{H, h}(x,y) =& \int_{\Rd} \frac{1-\norm{x-u}^{H+h}}{H+h} \psi(u,y) du + \int_{\Rd} \frac{1-\norm{y-v}^{H+h}}{H+h}  \psi(v,x) dv \\
&-\int_{\Rd}\int_{\Rd} \frac{1-\norm{u-v}^{H+h}}{H+h} \psi(u,y)\psi(v,x) dudv.
\end{aligned}
\ee
Then using (\ref{psiConditionOne}) along with \eqref{gg1}--\eqref{gg3} we get that  
\begin{equation*} 
\E(X^{H}(x)X^{h}(y)) = C_{H, h}\left( \frac{1- \norm{x-y}^{H+h}}{H+h} + g^{H,h}(x,y) \right), 
\end{equation*}
as needed. 

The boundedness of $g^{H,h}$ on $D\times D$ is a direct consequence of condition (\ref{psiConditionTwo}). 

In order to simplify the notation in what follows, we define  $$l^{h}(x)=(1-\norm{x}^h)h^{-1}$$
and denote by $\star$ the convolution operation. 

From \eqref{gg3} we get 
\be \label{g-conv}
\begin{aligned}
g^{H, h}(x,y)  =  (l^{H+h}\star\psi(\cdot,y))(x) +  (l^{H+h}\star\psi(\cdot,x))(y) + (l^{H+h}\star\psi(\cdot,x)\star\psi(\cdot,y))(0).
\end{aligned}
\ee

Next we deal with the family of fractional Brownian motions $\{\wt B^H\}_{H \in (0,H_0)}$ in (\ref{eq:mandelbrot_van_ness}).  The proof in this case follows the same lines. The only difference appears in $C_{H,h}$ and $g^{H,h}$ as we get that 
\begin{align*}
C_{H,h} = b_{h,H}\sqrt{\Gamma(H)\Gamma(h)}(H+h),
\end{align*}
and
\be \label{g-mv}
\begin{aligned}
g^{H,h}(x,y) =&  (l^{H+h}\star\psi(\cdot,y))(x) +  (l^{H+h}\star\psi(\cdot,x))(y) + (l^{H+h}\star\psi(\cdot,x)\star\psi(\cdot,y))(0) \\
&+\frac{o_{H,h}}{b_{H,h}(H+h)}\bigg(\int_{\R} \mathrm{sgn}(u)\abs{x-u}^{H+h} \psi(u,y) du \\
&+ \int_{\R} \mathrm{sgn}(v)\abs{y-v}^{H+h}\psi(v,x) dv \\
&-\int_{\R}\int_{\R} \mathrm{sgn}(v-u)\abs{u-v}^{H+h}\psi(u,y)\psi(v,x) dudv\bigg).
\end{aligned} 
\ee
Again the boundedness of $g^{H,h}$ on $D\times D$ is a direct consequence of (\ref{psiConditionTwo}).

\textbf{Step 3: Property \eqref{c-h} for $C_{H,h}$.} 
 We first prove \eqref{c-h} for $\{B^H\}_{H\in (0,H_0)}$ in (\ref{eq:one_dim_symmetric}) and (\ref{eq:fbm_construction_rd}).

From the definition of $c^{d}_{H,h}$ in Appendix \ref{appendix} and the definition of $C_{H,h}$ in \eqref{gg2} we get
\begin{align*}
C_{H,h} =& \sqrt{\Gamma(H)\Gamma(h)}\frac{\Gamma(\frac{H+h+1}{2})\sqrt{\Gamma(\frac{H+d}{2})H\Gamma(2H)\sin(\pi H)\Gamma(\frac{h+d}{2})h\Gamma(2h)\sin(\pi h)}}{\Gamma(\frac{H+h+d}{2})\Gamma(H+h)\sin(\frac{H+h}{2}\pi)\sqrt{\Gamma(H+\frac{1}{2})\Gamma(h+\frac{1}{2})}} \\
=& \frac{\sqrt{\pi}2^{h+H}\Gamma(1-\frac{H+h}{2})\Gamma(\frac{H+d}{2})\Gamma(\frac{h+d}{2})\sqrt{\Gamma(2H +1) \Gamma(2h + 1)}}{\sqrt{\Gamma(H+\frac{1}{2})\Gamma(h+\frac{1}{2})}\Gamma(\frac{H+h+d}{2})\sqrt{\Gamma(1-H)\Gamma(1-h)}},
 \end{align*}
where we used the duplication and reflection formula for the gamma function in the last quality.

Since the gamma function is smooth on the positive half-line away from zero, it follows that $C_{H,h}$ is smooth in $h,H \in [0, \frac{1}{2})$.
Since $\Gamma(\frac{1}{2}) = \sqrt{\pi}$, we get that $C_{0,0}=1$, and the statement about the convergence of $C_{H,h}$ in \eqref{c-h} follows.

The proof of \eqref{c-h} for the family of fractional Brownian motions $\{\wt B^H\}_{H \in (0,H_0)}$ in (\ref{eq:mandelbrot_van_ness}) follows by a similar argument.

\textbf{Step 4: property \eqref{eq:gEstimate} for $g^{H,h}$.}  We first prove \eqref{eq:gEstimate} for $\{B^H\}_{H\in (0,H_0)}$ in (\ref{eq:one_dim_symmetric}) and (\ref{eq:fbm_construction_rd}).

We will first identify the limiting function $g$  by using dominated convergence 
\be \label{g-fbf}
\begin{aligned}
g(x,y)&:=\lim_{H\to 0}g^{H,H}(x,y) \\
&= \int_{\Rd}\log\norm{x-u}\psi(u,y)du + \int_{\Rd}\log\norm{x-u}\psi(v,x)dv \\
&\quad -\int_{\Rd}\int_{\Rd}\log\norm{u-v} \psi(u,y)\psi(v,x)dudv
\end{aligned}
\ee
Note that the conditions of the dominated convergence theorem are satisfied since, 
\begin{align*}
\abs{\frac{1 - \norm{x-y}^{2H}}{2H}} \le& \log  \frac{1}{\norm{x-y}}\mathbf{1}_{\{\norm{x-y}<1\}} + \abs{1-\norm{x-y}} \mathbf{1}_{\{\norm{x-y}\ge1\}} \\
\le& (\log_-\norm{x-y})^{2} + 1 + \norm{x} + \norm{y} =: f(x,y).
\end{align*}

Then by the conditions (\ref{psiConditionOne}) -- (\ref{psiConditionFour}) it holds for all $y,x\in D$
\begin{align*}
\int_{\Rd}f(x,u)\psi(u,y) du  < \infty, \\
\int_{\Rd}\int_{\Rd}f(v,u)\psi(u,y)\psi(v,x) du dv  < \infty.
\end{align*}
This justifies the use of dominated convergence and also proves the boundedness of $g$ on $D\times D$.

In order to show \eqref{eq:gEstimate} we will use the following lemma, which is proved in the end of the section. 

\begin{lemma}\label{lem:log_approximation}
For any $0\leq h \leq 1$ we have 
\begin{align*}
0 \le\; \log\frac{1}{r} - \frac{1-r^{h}}{h} \;\le 
\left\{ \begin{array}{cc}
\frac{h}{2}\log^{2}(r) & 0 < r\le 1, \\
h(r - 1 - \log(r)) & r > 1.
\end{array}\right.
\end{align*}
\end{lemma}

For $r>0$ and $h\in[0,1]$ define 
\begin{align} \label{eta}
\eta(h, r) = \log\frac{1}{r} - \frac{1 - r^{h}}{h}.
\end{align}
Then from the definitions of $g^{H,h}$ and $g$ in in \eqref{gg3} and \eqref{g-fbf}, respectively, we have   
\begin{align*}
&g(x,y) - g^{H, h}(x,y) \\
&= - \int_{\Rd} \eta(H+h, \norm{u-x})\psi(u,y)du  - \int_{\Rd} \eta(H+h, \norm{v-y})\psi(v,x)dv \\
&\quad + \int_{\Rd}\int_{\Rd} \eta(H+h, \norm{u-v})\psi(u,y)\psi(v,x)du dv.
\end{align*}
From Lemma \ref{lem:log_approximation} we have for all $x,y \in D$ and $h,H \in (0,\frac{1}{2})$,
\begin{align*}
0 \le \eta(H+h, \norm{x-y}) \le C (H+h)\left(\left(\log_-\norm{x-y}\right)^2 + \norm{x} + \norm{y}\right).
\end{align*}

Hence we have from (\ref{psiConditionTwo})--(\ref{psiConditionFour}) we get that 
\begin{align*}
\sup_{x,y \in D} |g(x,y) - g^{H,h}(x,y)| \le C(H+h), 
\end{align*}
 and we proved \eqref{eq:gEstimate} for $\{B^H\}_{H\in (0,H_0)}$ in (\ref{eq:one_dim_symmetric}) and (\ref{eq:fbm_construction_rd}). 
 
 Finally we prove \eqref{eq:gEstimate} for $\{\wt B^H\}_{H \in (0,H_0)}$ in (\ref{eq:mandelbrot_van_ness}).
 
First we notice that the function $g$ in this case is similar to $g$ in \eqref{g-fbf}. Indeed  $o_{H,H} = 0$ in \eqref{g-mv} for all $0< H < H_0$, therefore $g^{H,H}$ in  \eqref{g-mv} is  identical to \eqref{gg3}, which convergence towards the function $g$

Note that the first part of $g^{H,h}$ in  \eqref{g-mv} is identical to $g^{H,h}$ in \eqref{g-conv}, for which we have proved \eqref{eq:gEstimate}.  
It follows that we only need to show that  
\begin{align*}
T^{H,h}(x,y):=\frac{o_{H,h}}{b_{H,h}(H+h)}\bigg(&\int_{\R} \mathrm{sgn}(u)\abs{x-u}^{H+h} \psi(u,y) du \\ 
&+ \int_{\R} \mathrm{sgn}(v)\abs{y-v}^{H+h}\psi(v,x) dv \\
&-\int_{\R}\int_{\R} \mathrm{sgn}(v-u)\abs{u-v}^{H+h}\psi(u,y)\psi(v,x) dudv\bigg),
\end{align*}
converges to zero uniformly in the same sense.

Note from (\ref{psiConditionTwo}) we get 
$$
\sup_{x,y \in D} |T^{H,h}(x,y)| \leq C \abs{\frac{o_{H,h}}{b_{H,h}(H+h)}}. 
$$
Them statement then follows by estimating the multiplicative factor 
\begin{align*}
\abs{\frac{o_{H,h}}{b_{H,h}(H+h)}}=\abs{\frac{\sin((h-H)\frac{\pi}{2})\sin((h+H)\frac{\pi}{2})}{\cos((h-H)\frac{\pi}{2})\cos((h+H)\frac{\pi}{2})(H+h)}} \le C (H+h), 
\end{align*}
for all $0< h, H < H_0.$   
\qed
\begin{proof}[Proof of Lemma \ref{lem:log_approximation}]
Let $\eta(h,r)$ as in \eqref{eta}. Note that the lower bound $\eta(h,r)\ge 0$ is equivalent to $e^{h\log(r)} -1 \ge h\log(r)$, which is obviously true.

Next, we prove the upper bound on $\eta(h,r)$. Taylor's theorem applied to the function  $h \mapsto 1 - r^{h}$ at $h=0$ yields
\begin{equation}\label{lim:remainder_l_expansion}
\lim_{h\to0}\; \frac{\eta(h, r)}{h} = \frac{1}{2} \log^{2}(r).
\end{equation}
We further analyse the behaviour of $\frac{\eta(h, r)}{h}$.
Differentiation in $h$ yields
\begin{align*}
\gamma(h,r) := \frac{\partial}{\partial h}\Big(\frac{\eta(h, r)}{h} \Big)
&= \frac{\partial}{\partial h}\Big(\frac{-h\log(r) - (1-r^{h})}{h^{2}}\Big) \\
&= \frac{h\log(r)\big(e^{h\log(r)} + 1\big) + 2\big(1-e^{h\log(r)} \big)}{h^3}.
\end{align*}
We observe that for any $h>0$ and $0<r\leq 1$ we have 
$$
  \gamma(h,r) \le  \frac{h\log(r)2 + 2(-h\log(r))}{h^3} \leq 0.
$$  
Hence the convergence in (\ref{lim:remainder_l_expansion}) is monotone when $0<r\le1$, and in particular
\begin{equation}
0 \le \frac{\eta(h, r)}{h} \le \frac{1}{2} \log^{2}(r), \for 0<r\le1.
\end{equation}

Also for any $h>0$ and $r>1$
$$
  \gamma(h,r) = -e^{h\log(r)} \gamma(h,r^{-1}) \ge 0.
$$
We obtain the bound for $h>0$
$$0 \le \frac{\eta(h, r)}{h} \le \frac{\eta(1, r)}{1} \le -\log(r) + r - 1, \for \  r>1.$$
\end{proof}

\appendix
\section{List of Constants} \label{appendix} 
In order to ease the notation, we summarized the constants that appear in Sections \ref{corr-comp} and \ref{sec:proof_normalization} in the following list
\begin{align*}
m_H &= \frac{\sqrt{\Gamma(2H+1)\sin(\pi H)}}{\Gamma(H+\frac{1}{2})} \\
a_{h,H} &= \frac{1}{\pi}\sqrt{\Gamma(2h +1)\sin(\pi h)} \sqrt{\Gamma(2H +1)\sin(\pi H)} \Gamma(-(h+H)) \\
b_{h,H} &= a_{h,H}\cdot\cos\Big((h-H)\frac{\pi}{2}\Big)\cos\Big((h+H)\frac{\pi}{2}\Big) \\
o_{h,H} &= a_{h,H}\cdot\sin\Big((h-H)\frac{\pi}{2}\Big)\sin\Big((h+H)\frac{\pi}{2}\Big) \\
k^{d}_H & = \sqrt{\frac{\Gamma(\frac{d}{4}-\frac{H}{2})^{2}\Gamma(H+\frac{d}{2})H\Gamma(2H)\sin(\pi H)}{2^{2H+1}\pi^{\frac{d+1}{2}}\Gamma(\frac{d}{4}+\frac{H}{2})^{2}\Gamma(H+\frac{1}{2})}}\\
c^{d}_{H,h} &= 2\pi^{\frac{d+1}{2}}\frac{k_H^{d}k_h^{d}}{\Gamma(\frac{d}{4}-\frac{H}{2})\Gamma(\frac{d}{4}-\frac{h}{2})} \frac{\Gamma(\frac{H+h+1}{2})}{\Gamma(\frac{H+h+d}{2})}  \frac{2^{H+h}\Gamma(\frac{d}{4}+\frac{H}{2})\Gamma(\frac{d}{4}+\frac{h}{2})}{(H+h)\Gamma(H+h)\sin(\frac{H+h}{2}\pi)}\\
&=  \frac{\Gamma(\frac{H+h+1}{2})\sqrt{\Gamma(\frac{H+d}{2})H\Gamma(2H)\sin(\pi H)\Gamma(\frac{h+d}{2})h\Gamma(2h)\sin(\pi h)}}{\Gamma(\frac{H+h+d}{2})(H+h)\Gamma(H+h)\sin(\frac{H+h}{2}\pi)\sqrt{\Gamma(H+\frac{1}{2})\Gamma(h+\frac{1}{2})}}
\end{align*}

\section*{Acknowledgments}
We are very grateful to Nathanael Berestycki whose numerous useful comments enabled us to significantly improve this paper.

\end{document}